\documentclass{amsart}
\usepackage{amssymb,amsxtra}

\newtheorem{lemma}{Lemma}[section]
\newtheorem{thm}[lemma]{Theorem}
\newtheorem{prop}[lemma]{Proposition}
\newtheorem{cor}[lemma]{Corollary}

\newcommand{\arr}[2]{\begin{array}{#1}#2\end{array}}

\newcommand{\matr}[2]{\left(\begin{array}{#1}#2\end{array}\right)}
\newcommand{\slrw}[1]{\stackrel{#1}{\longrightarrow}}
\newcommand{\hmm}[2]{\mathrm{Hom}_{#1}(#2)}

\begin{document}
\title{Coverings and Truncations of Graded Self-injective Algebras}
\author[Guo]{Jin Yun Guo}
\email{Guo:gjy@xtu.edu.cn}
\address{Jin Yun Guo\\ Department of Mathematics, Xiangtan University, Xiangtan,  CHINA}
\thanks{This work is partly supported by
Natural Science Foundation of China \#10971172}
\dedicatory{}



\date{}

\begin{abstract}
Let $\Lambda$ be a graded self-injective algebra. We describe its smash product $\Lambda\# k\mathbb Z^*$ with the group $\mathbb Z$, its Beilinson algebra and their relationship.
Starting with $\Lambda$, we construct algebras with finite global dimension, called $\tau$-slice algebras,
we show that their trivial extensions are all isomorphic, and their repetitive algebras are the same  $\Lambda\# k\mathbb Z^*$.
There exist $\tau$-mutations similar to the BGP reflections for the $\tau$-slice algebras.
We also recover Iyama's absolute $n$-complete algebra as truncation of the Koszul dual of certain self-injective algebra.
\end{abstract}

%\keywords{Graded self-injective algebra; Nakayama translation; trivial extension; $\tau$-mutation; $\tau$-slice algebra; $m$-cubic algebra}

\maketitle

\section{Introduction}
In \cite{bgg}, it is proved that the derived category $D(\mathrm{coh}\,\mathbb P^{n-1})$ of the coherent sheaves of a projective space is equivalent to the stable category $\underline{\mathrm{gr}}\wedge V$ of the exterior algebra (called BGG correspondence).
Koszul duality between Artin-Schelter regular Koszul algebra and the self-injective Koszul algebra \cite{s, m2} generalizes BGG correspondence to non-commutative setting \cite{ms, mor, jo, hw}.
It is also known that the derived category of projective line is equivalent to  the derived category of a Kronecker algebra, which is hereditary and of finite global dimension \cite{lz}.
Recently, Chen  proves in \cite{cx} that for a well graded self-injective algebra, the category of its graded  modules is equivalent to  the category of the graded modules over the trivial extension algebra of its Beilinson algebra, which is of finite global dimension.
As a consequence, the derived category of its Beilinson algebra is equivalent to the stable category of its trivial extension \cite{cx}.
So the case of Kronecker algebra can be generalized, and the BGG correspondence is extended to a derived category of algebra of finite global dimension, and we get equivalences of triangulated categories as follows,
$$D^b(\mathrm{coh} X) \simeq \underline{\mathrm{gr}} \Lambda \simeq D^b(\Lambda'),$$
here the left side is the bounded derived category of the quasi-coherent sheaves of  non-commutative projective space, middle is the stable category of a graded self-injective algebra $\Lambda$ and the right side is the bounded derived category of an algebra $\Lambda'$ of finite global dimension.
According to \cite{cx}, the right equivalence is also well known \cite{h1, h2}, when we started with $\Lambda'$.
It follows from \cite{h1} that  $D^b(\Lambda')$ is also equivalent to $\underline{\bmod}\widehat{\Lambda'}$, the stable category of finite generated modules over the repetitive algebra $\widehat{\Lambda'}$ of ${\Lambda'}$, as triangulate categories.

This paper mainly studies the algebras appearing in the right side of the above equivalences of the triangulated categories.
Starting with a graded self-injective algebra $\Lambda$, we are interesting in the algebras  $\Lambda'$ of finite global dimension.
Our approach is similar to the classical approach initiated in \cite{hwu}, and developed, e.g., in \cite{sy, oty}, using repetitive algebras and coverings in study self-injective algebra.
Our aim is to find out how to construct the algebras $\Lambda'$ of finite global dimension from $\Lambda$, and how such algebras are related. Coverings and truncations related to the Nakayama functor play key roles in our approach.
We also show by example how Iyama's higher representation theory is related to our constructions.

\medskip

Throughout this paper we assume that $k$ is an algebraically closed field.

Let $\Lambda$ be a self-injective algebra over $k$ and let $\mathcal N = D\mathrm{Hom}_{\Lambda}(\quad, \Lambda)$ be the Nakayama functor.
$\mathcal N$ is an auto-equivalence on the category of $\Lambda$-modules, and it induces a permutation $\tau$ on the vertex set of the Gabriel quiver of $\Lambda$, which we call {\em Nakayama translation}.
When the group $G$ generated by the Nakayama translation acts freely on the vertices, $\Lambda$ is a regular covering of its orbit algebra with respect to the group $G$. This orbit algebra is a weakly symmetric algebra, that is, graded self-injective with trivial Nakayama translation.

To find $\Lambda'$ when a graded self-injective algebra $\Lambda$ is given, we first go to the smash product $\Lambda \# k\mathbb Z^*$ of $\Lambda$ with the group $\mathbb Z$, whose bound quiver is described as the separated directed quiver of $\Lambda$.
The Beilinson algebra of $\Lambda$ defined in \cite{cx} is the first candidate for $\Lambda'$.
We describe the bound quiver of the Beilinson algebra of $\Lambda$ as some truncation of the bound quiver of $\Lambda\#k \mathbb Z^*$.
We also show that the orbit algebra of $\Lambda\# k\mathbb Z^*$ with respect to the Nakayama translation is a twisted trivial extension of the Beilinson algebra and that $\Lambda\# k\mathbb Z^*$ is exactly the repetitive algebra of the Beilinson algebra.

There are more such algebras  $\Lambda'$ for a given graded self-injective algebra $\Lambda$, in addition to the Beilinson algebra.
These algebras are obtained by truncate the bound quiver of  $\Lambda\# k\mathbb Z^*$ mimic the complete slice in the tilting theory \cite{hr, hwu}.
We call the truncated bound quivers complete $\tau$-slices and the  algebras obtained $\tau$-slice algebras.
We also introduce the $\tau$-mutations for the $\tau$-slices algebras, mimic the Bernstein-Gelfand-Poromarev(BGP) reflections.
Each connected component of the Beilinson algebra defined in \cite{cx} is a $\tau$-slice algebra.
We show all the $\tau$-slice algebras have equivalent derived categories, by showing that they have the same trivial extensions, and the same repetitive algebra, $\Lambda\# k\mathbb Z^*$.

The $\tau$-mutation is not a direct generalization of the BGP reflection, we need to go to the Koszul dual for it.
We show that $\tau$-slice algebras are Koszul whenever $\Lambda^T$ is.
If $\Lambda$ is of Loewy length $3$, all the $\tau$-slice algebras are Koszul and the $\tau$-mutation induces the BGP reflection on their Koszul dual.

Using Nakayama translations in the place of higher Auslander-Reiten translation, we obtained Iyama's absolute $n$-complete algebra as a truncation of the Artin-Schelter regular algebra associated to some McKay quiver\cite{gum} (called  $m$-cubic algebra).
Here we need the characteristic of $k$ to be $0$.
Starting with a finite cyclic group $\mathbb Z/(r+1)\mathbb Z$ of order $r+1$ as a subgroup of $k^*$, we construct by covering and embedding to get a direct sum $(\mathbb Z/(r+1)\mathbb Z)^{m}$ of $m$ copies of cyclic group $\mathbb Z/(r+1)\mathbb Z$ in $\mathrm{GL}(m, k)$, whose McKay quiver has the form a netted $m$-ball, and locally, the Nakayama translation is related to the diagonal of  $m$-cubes.
We first embed $(\mathbb Z/(r+1)\mathbb Z)^{m-1}$ with $\mathrm{GL}(m-1,k)$ into  $\mathrm{SL}(m,k)$, to make Nakayama translation returning arrows opposite to the maximal bound paths at each vertex in the McKay quiver.
Then use covering to extend the group to a subgroup $(\mathbb Z/(r+1)\mathbb Z)^{m}$ in $\mathrm{GL}(m,k)$ with non-trivial Nakayama translation in its  McKay quiver(see \cite{gcv}).
This new Nakayama translation coincides with the $m$-Auslander-Reiten translation in \cite{iy3}, when we truncate out Iyama's absolute $m$-complete algebra from the $m$-cubic algebra.
We believe that graded self-injective algebra provided certain framework for study the higher representation theory, and those with finite complexity behavior like tame algebras.

The paper is organized as follow.
In Section 2,  we recall basic notions on bound quivers, path algebras etc., to fix terminology.
We also give a description on the graded self-injective algebra of Loewy length $l+1$, using bound quiver with Nakayama translation.
In Section 3, we study the orbit algebra of a graded self-injective algebra $\Lambda$ with respect to the Nakayama functor.
We show that the orbit algebra with respect to the Nakayama functor is self-injective and $\Lambda$ is a covering of the orbit algebra when the group generated by the Nakayama functor acts freely.
In section 4, we introduce separated directed quiver $\overline{Q}$ for the bound quiver $Q$ of a graded self-injective algebra $\Lambda$.
We show that this quiver is the bound quiver of the smash product $ \Lambda\# k\mathbb Z^*$ and the group generated by the Nakayama functor acts freely on the indecomposable projective $ \Lambda\# k\mathbb Z^*$-modules.
We also introduce special truncated quiver of the separated directed quiver, and discuss some basic properties of these quivers.
In Section 5, We show that the bound quiver of the Beilinson algebra of $\Lambda$ is the total special truncated quiver.
We also show that the orbit algebra of $ \Lambda\# k\mathbb Z^*$ with respect to the group generated by the Nakayama functor is isomorphic to a twisted trivial extension of the Beilinson algebra, and that $ \Lambda\# k\mathbb Z^*$ is the repetitive algebra of the Beilinson algebra of $\Lambda$.
In Section 6, we introduce $\tau$-slices, $\tau$-slice algebras and $\tau$-mutations.
The action of the $\tau$-mutations on the $\tau$-slices is transitive.
We prove that the trivial extension of the $\tau$-slice algebras are isomorphic when they are  $\tau$-mutations one another.
Such isomorphism induces equivalence of the derived categories of the $\tau$-slice algebras, and $\tau$-mutation can be regarded as generalization of tilting process.
We also show that all the $\tau$-slice algebras have the same repetitive algebra, $ \Lambda\# k\mathbb Z^*$.
In the Section 7, we prove that all the $\tau$-slice algebras are Koszul if we start with a Koszul self-injective algebra, and for a self-injective algebra with vanishing radical cube,  the Yoneda algebra of the $\tau$-mutation of a  $\tau$-slice algebra is exactly the  BGP reflection of its Yoneda algebra.
In the last section, we show that Iyama's absolute $m$-complete algebra is a truncation of an $m$-cubic algebra, which is the Koszul dual of the skew group algebra of direct sum of $m$ copies of a cyclic group over exterior algebra.

\section{Preliminaries}

Throughout this paper, all the algebras are basic and all the modules are left modules, when not specialized.
By a quiver we usual mean a bound quiver $Q = (Q_0, Q_1, \rho)$, that is, a quiver with  the vertex set $Q_0$, the arrow set $Q_1$  and the relation set $\rho$.
$\rho$ is a set of linear combinations of paths of length larger or equal to $2$.
We use the same notation $Q$ for both the quiver and the bound quiver, denote by $kQ$ the path algebra of $Q$ and by $k(Q)= kQ/(\rho)$ the algebra given by the bound quiver $Q$, that is, the quotient algebra of the path algebra of $Q$ modulo the ideal generated by the relations.
A path is called a {\em  bound path} if its image in $k(Q)$ is nonzero.

By Gabriel's theorem, $\Lambda$ is given by its bound quiver $Q_{\Lambda} = (Q_0, Q_1, \rho)$, that is, $\Lambda \simeq kQ/(\rho)$.
An algebra $\Lambda$  is called a graded  algebra in this paper if $\Lambda = \Lambda_0 +\Lambda_1 + \Lambda_2 +\cdots$ as a direct sum of vector spaces. $\Lambda_0 $ is semi-simple basic algebra, a direct product of $|Q_0|$-copies of $k$, and $\Lambda_i \Lambda_j = \Lambda_{i+j}$.
Write $\mathbf r = \Lambda_1 + \Lambda_2 +\cdots$.
Let $\{e_i| i\in Q_0\}$  be a complete set of orthogonal idempotents of $\Lambda$, then $1=\sum_{i\in Q_0}e_i$.
Let  $E(\Lambda) = \mathrm{Ext}_{\Lambda}(\Lambda_0, \Lambda_0)$ be its Yoneda algebra.

We now characterize the graded self-injective algebras  using bound quivers.
We say that a bound quiver $Q=(Q_0,Q_1,\rho)$ is {\em homogeneous} provided that each of the paths appearing in a given linear combination of $\rho$ has the same length.
Two relation sets $\rho$ and $\rho'$ of a given quiver $Q$ is said to be equivalent if they generate the same ideal in the path algebra $kQ$.
In this case, we also say that two bound quivers $(Q_0,Q_1,\rho)$ and $(Q_0,Q_1,\rho')$ are equivalent.
Clearly, equivalent bound quivers define isomorphic algebras.

Since path algebra is graded, we obvious have the following proposition.

\begin{prop}\label{grada}
$k(Q)$ is a graded algebra if and only if $Q$ is equivalent to a homogeneous bound quiver.
\end{prop}

Fix an integer $l \ge 1$, a homogeneous bound quiver $Q$ is said to  be {\em stable of Loewy length} $l+1$ if there is a permutation $\tau$ on the vertex set of the quiver, such that the following conditions are satisfied.
\begin{enumerate}
  \item The maximal bound paths of $Q$ have the same length $l$;
  \item For each vertex $i$, there is a maximal bound path from $\tau i$ to $i$;
  \item There is no bound path of length $l$ from $\tau i$ to $j$ for any $j \neq i$;
  \item All the maximal bound paths starting at the same vertex are linearly dependent.
\end{enumerate}
The permutation $\tau$ is called the {\em Nakayama translation} of the stable bound quiver $Q$.

We have the following theorem characterizing the bound quiver  of a graded self-injective algebra.

\begin{thm}\label{iffsi}
Let $\Lambda=k(Q)$ be the algebra given by a bound quiver  $Q$, then $\Lambda$ is a graded self-injective algebra with Loewy length $l+1$ if and only if $Q$ is equivalent to a stable bound quiver of Loewy length $l+1$.
\end{thm}

\begin{proof}
Assume that $\Lambda=k(Q)$ is the algebra given by a stable bound quiver $Q$ of Loewy length $l+1$ with Nakayama translation $\tau$.
$\mathbf r$ is its radical, and we have that $\mathbf r^{l+1} =0$ since maximal bound paths in $Q$ have the same  length $l$.
Let $e_i$ be the idempotent corresponding to the vertex $i$, and let $S_i$ be the simple module corresponding to $i$, $P(i)$ be its projective cover  and $I(i)$ be its injective envelope.
For maximal bound path $p$, $\mathbf r p = 0$, so they span the socle of $\Lambda$. Since maximal bound paths starting at $i$ end at $\tau^{-1} i$, and they are all linearly dependent, we see that the socle of $P_i \simeq \Lambda e_i$ is isomorphic to $S_{\tau^{-1} i}$.
So each indecomposable projective has a simple socle and $\mathrm{soc\,} P_i \simeq S_{\tau^{-1} i}= P_{\tau^{-1} i} /\mathbf r P_{\tau^{-1} i}$.
By considering the maximal bound paths ending at $i$, one gets that indecomposable injective has a simple top and $I_i/ \mathbf r I_i \simeq S_{\tau i}$.
This implies that $P_i \simeq I_{\tau^{-1} i}$ for each $i \in Q_0$, and $\Lambda$ is a graded self-injective algebra of  Loewy length $l+1$.

Let $Q=(Q_0,Q_1,\rho)$ be a bound quiver of a graded self-injective algebra $\Lambda$.
We may assume that $Q$ is homogeneous.
Let $\tau$ be the permutation of $Q_0$ induced by the Nakayama functor of $\bmod \Lambda$.
Then $\tau$ sends each vertex $i$ to the vertex $\tau i$ corresponding to the top of the injective envelope of the simple $S_i$.
Since projectives have the same Loewy length, say $l$, by \cite{m2}, so the maximal bound paths of $Q$ have the same length $l$.
Since an indecomposable projective with top $S_i$ is the indecomposable injective with socle $S_{\tau^{-1} i}$  for each vertex $i$, there is a maximal bound path from $\tau i$ to $i$ which is a multiple of each bound path of length $l$ from $\tau i$ to $i$.
We also see that there is no bound path of length $l$ from $\tau i$ to $j$ for any $j \neq
i$.
This shows that $Q$ is a stable bound quiver of Loewy length $l+1$.
\end{proof}

So the Nakayama translation is induced by the Nakayama functor.

A {\em walk} from a vertex $i$ to a vertex $j$ in a quiver is a sequence of  paths $p_1,\ldots,p_r$ in $Q$ satisfies the following conditions:
\begin{enumerate}
 \item $r$ is odd, the length $l(p_t)>0$ for $t=2,\ldots, \frac{r-1}{2}$, (here we allow the length of $p_1$ and $p_r$ to be zero);
 \item $i$ is the starting vertex of $p_1$ and $j$ is the ending  vertex of $p_r$, $p_{2t}$ and $p_{2t+1}$ have the same starting vertex and $p_{2t-1}$ and $p_{2t}$ have the same ending vertex.
\end{enumerate}
When all the paths in a walk are bound paths, we call this walk a {\em bound walk}

Now consider a stable bound quiver.
By embedding a bound path in maximal ones, one sees that for a bound path from vertex $i$ to vertex  $j$, there is a bound path from $j$ to $\tau^{-1} i$ and a bound path from $\tau j$ to $i$.
For a walk in a stable bound quiver with two nontrivial bound paths starting from vertices $i$ and $i'$ and ending at vertex $j$, there is a walk with two nontrivial bound paths from $\tau j$ to $i$ and $i'$.
Dually, for a walk with two nontrivial paths starting at vertex $i$ and ending at vertices $j$ and $j'$, there is a bound walk with two nontrivial paths from $ j$ and $j'$ to $\tau^{-1} i$.
In a  finite connected stable bound quiver, $\tau$ is periodic, we have the following lemma.

\begin{lemma}\label{path}
Let $Q$ be a finite connected stable bound quiver.
For each pair of vertices $i$ and $i'$  which are connected by a walk in the stable bound quiver, there is an (unbound) path in $Q$ from $i$ to $i'$.
\end{lemma}

We will need the following results, see Lemma 2.1 of \cite{g4}.

\begin{lemma}\label{bil}
Let $\Lambda$ be a self-injective algebra, let $Q$ be its bound quiver and $\tau$ be the Nakayama translation on $Q$.
Then for any $i,j \in Q_0$ if $e_j \Lambda e_i \neq 0$, we have a non-degenerate bilinear form $e_{\tau^{-1} i} \Lambda e_j \otimes e_j\Lambda e_i \to k$ satisfying the multiplicative property, that is $(xy,z) = (x,yz)$.

If $\Lambda$ is graded with Loewy length $l+1$, the bilinear form restricts to a non-degenerate one on $e_{\tau^{-1} i} \Lambda_{l-t} e_j \otimes e_j\Lambda_t e_i \to k$ for $0 \le t \le l$, whenever $e_j\Lambda_t e_i \neq 0$.
\end{lemma}

So we have that $e_{\tau^{-1} i} \Lambda_{l-1} e_j \simeq D e_j\Lambda_1 e_i$ and $e_{j} \Lambda_{1} e_i \simeq \tau e_j\Lambda_1\tau e_i$ as vector spaces.
As a corollary of Lemma ~\ref{bil}, we have the following result.
\begin{cor}\label{arrows}
Let $\Lambda$ be a graded self-injective algebra with Loewy length $l+1$ and Nakayama translation $\tau$, let $Q$ be its quiver.
Then for any $i,j$ in $Q_0$, we have:
\begin{enumerate}
\item The number of arrows from $i$ to $j$ is $\mathrm{dim\,}_k \, e_{\tau^{-1} i} \Lambda_{l-1} e_j$;
\item The number of arrows from $i$ to $j$ and the number of arrows from $\tau i$ to $\tau j$ are the same.

\end{enumerate}
\end{cor}

\section{Nakayama Translation and Orbit Algebra}
Following \cite{arsc}, a $k$-additive category $\mathcal C$ is called {\em locally bounded}, if it satisfies the following conditions:
\begin{enumerate}
 \item For each object $P$ in $\mathcal C$, $\mathrm{End}(P)$ is local;
 \item For each pair  $P, P'$ of objects in $\mathcal C$, $\mathrm{dim}_k\,\mathrm{Hom}(P,P')$ is finite;
 \item Distinct objects in $\mathcal C$ are non-isomorphic;
 \item  For each object $P$ in $\mathcal C$, there are only finite many  object $P'$ in $\mathcal C$ such that $\mathrm{Hom}(P,P')\neq 0$ or $\mathrm{Hom}(P',P)\neq 0$.
\end{enumerate}

Let $\mathcal{C,D}$ be locally bounded categories.
A $k$-linear functor $F:\mathcal{C} \to \mathcal{D}$ is called a {\em covering functor} if,
\begin{enumerate}
 \item $F$ is surjective on objects,
 \item For each object $P$ in $\mathcal C$, $F$ induces isomorphisms $$\bigoplus_{P'\in F^{-1}(Q)}\mathrm{Hom}(P',P) \to \mathrm{Hom}(Q,F(P)) $$ and $$\bigoplus_{P'\in F^{-1}(Q)}\mathrm{Hom}(P,P') \to \mathrm{Hom}(F(P),Q).$$
\end{enumerate}
If a group $G$ of $k$-automorphisms on $\mathcal C$ acts freely on the objects and  $\mathcal D$ is equivalent to its orbit category, then we call the covering $F$ regular (or Galois) and call $G$ its group.

Let $\Lambda$ and $\Lambda'$ be two $k$-algebras.
If there are locally bounded categories $\mathcal{C,D}$ such that with the naturally defined multiplications
$$\Lambda \simeq \bigoplus_{P, P'\in \mathcal C} \mathrm{Hom}(P,P') \quad \mbox{and } \quad \Lambda' \simeq \bigoplus_{D, D'\in \mathcal D} \mathrm{Hom}(D,D') $$
as algebras and there is a covering functor $F: \mathcal{C} \to \mathcal{D}$, then we say that $F$ is {\em a covering} from $\Lambda$ to $\Lambda'$.
If $F$ is regular with group $G$, we called covering from $\Lambda$ to $\Lambda'$ regular with group $G$.

Let $\mathcal P$ be a small $k$-additive category.
Let $G$ be a group of autofunctors of $\mathcal P$ and let $M$ be an  object in $\mathcal P$.
Assume that $G$ acts freely on the objects.
Define the {\em orbit algebra} $O(G, M)$ of $M$ with respect to $G$ to be the vector space
$$O(G, M)=\bigoplus_{F \in G} \mathrm{Hom}_{} (F M, M)$$
with the multiplication defined as follows:
For $f \in \mathrm{Hom}_{} (F M, M)$, $g \in \mathrm{Hom}_{} (F' M, M)$, with $F, F' \in G$,
$$f \cdot g = f  \circ  F g.$$
$O(G, M)$ is an associative $k$ algebra.

Let $\Lambda$ be a  self-injective algebra and let $\mathcal N$ be the Nakayama functor on $\bmod \Lambda$.
$\mathcal N$ restricts to an auto-equivalence on the category $\mathcal P=\mathcal P(\Lambda)$ of projective $\Lambda$-modules and induces an automorphism $\nu$ on $\Lambda$.

Let $G$ be the group generated by $\mathcal N$, and assume that $G$ acts freely on the objects in $\mathcal P $. An object $P$ of $\mathcal P$ is called {\em basic $G$-orbit generator} if  its indecomposable summands are taken from different orbits and we have $\mathrm{add}\, G P = \mathcal P$ for its orbit $G P =\{\mathcal N^t P| t\in \mathbb Z\}$.
Let $P=\bigoplus_{i\in I} P_i$ be a basic $G$-orbit generator.
Clearly, $O(G, P)=\mathrm{End} P$ when $G$ is trivial.
Denote by $\Lambda^T=O(G, P)$ the orbit algebra of $P$ with respect to $G$.

The following theorem tells us that $\mathcal N$ induces a covering from $\Lambda$ to the twisted orbit algebra $\Lambda^T = O(G, P)$.

\begin{thm}\label{covA} Assume that $G$ acts freely on the indecomposable objects in $\mathcal P $. Then $\mathcal N $ induces a regular covering  $N$ from $ \Lambda $ to $\Lambda^T$ with the group $G$, $\Lambda^T$ is a graded self-injective algebra whose Nakayama functor induces trivial Nakayama translation on the bound quiver.
\end{thm}

\begin{proof}
Let $\mathcal C =\mathrm{ind} \mathcal P$ be the category with objects the non-isomorphic indecomposable projective $\Lambda$-modules, with usual sets of $\Lambda$-maps as the sets of morphisms.
Let $P = \sum_{i\in I} P_i$ be a basic $G$-orbit generator of $\mathcal P(\Lambda)$, with $P_i$ indecomposable.
Then $\{P_i  | i\in I\}$ is a set of chosen representatives of the $G$-orbits of the indecomposable projective $\Lambda$-modules.
$$
\arr{lcl}{
\mathrm{Hom}_{\Lambda} (\Lambda, \Lambda)
&=&\bigoplus_{\mathcal N^t , \mathcal N^{t'} \in G} \bigoplus_{i,j \in I}\mathrm{Hom}_{\Lambda} ( \mathcal N^t P_i,\mathcal N^{t'} P_j)
}.$$

Let $\mathcal D$ be the category with objects the $G$-orbits $\{[P_i]|i \in I\}$ of indecomposable projective $\Lambda$-modules, with hom-sets $$\arr{lll}{\mathrm{Hom}([ P_j],[  P_i]) &=& \bigoplus_{P'' \in [ P_j]} \mathrm{Hom}_{\Lambda} ( P'',  P_i)\\ &=& \bigoplus_{\mathcal N^t  \in G} \mathrm{Hom}_{\Lambda} ( \mathcal N^t P_j,  P_i).}$$
The composition of the morphisms is defined as follows:
\begin{equation}f \cdot g =   f \circ  \mathcal N^t g  \end{equation}
for $f \in \mathrm{Hom}_{\Lambda} ( \mathcal N^t P_j, P_i), g \in  \mathrm{Hom}_{\Lambda} ( \mathcal N^{t'} P_k, P_j)$.
Clearly, $\mathrm{Hom}([ P_j],[  P_i]) \simeq \bigoplus_{\mathcal N^t \in G} \mathrm{Hom}_{\Lambda} (P_j,  \mathcal N^t  P_i)$.
Define
$$N(f) =  \mathcal N^{-t'} f  $$
for $f \in  \mathrm{Hom}_{\Lambda} (\mathcal N^t P_j,\mathcal N^{t'} P_i)$.
This is a $k$-homomorphism from $\mathrm{Hom}_{\mathcal D} (\mathcal N^t P_j,\mathcal N^{t'} P_i)$ to $\mathrm{Hom}_{\mathcal D} ([ P_j],[ P_i])$, and it is  easy to see that $N$ is a functor from $\mathcal C$ to $\mathcal D$. So $N$ is a covering from  $\mathcal C$ to $\mathcal D$.

It follows directly from the definition that $$\arr{ccl}{\bigoplus_{ \mathcal N^t  \in G}  \bigoplus_{i,j \in I}  \mathrm{Hom}_{\Lambda} ( \mathcal N^t P_j, P_i) &=&\bigoplus_{\mathcal N^t \in G} \mathrm{Hom}_{\Lambda} ( \mathcal N^t \bigoplus_{i \in I} P_i, \bigoplus_{i \in I} P_i)\\ &=&\bigoplus_{\mathcal N^t \in G} \mathrm{Hom}_{\Lambda} (  \mathcal N^t P,P)\\ &\simeq&  \mathcal O(G,P) = \Lambda^T  ,}$$
by comparing their composition laws.
So $\Lambda$ is a regular covering of $\Lambda^T $ with group $G$.

Write $e_{P}$ for the identity for an object $P$, then $P= \Lambda e_{P}$ for an indecomposable object in $\mathcal C$, and $[P]= \Lambda^T e_{[P]}$ for an indecomposable object in $\mathcal D$.
By Lemma ~\ref{bil}, $\mathrm{Hom}_{\Lambda} (P', P)_t = e_{P'} \Lambda_t e_{P} \simeq D e_{\mathcal N P}  \Lambda_{l-t} e_{P'} = D\mathrm{Hom}_{\Lambda} (\mathcal N P, P')_{l-t}$ for any $P, P'$, here $\mathrm{Hom}_{\Lambda} (P, P')_t$ denote the degree $t$ homomorphisms.
For an indecomposable $\Lambda^T$ module $\Lambda^T e_{[P]}$,
$$\arr{lll}{\Lambda^T e_{[P]}& =& \bigoplus_{[P']} \mathrm{Hom}_{\Lambda^T} (\Lambda^T,  \Lambda^T e_{[P]}) \\
&=&  \bigoplus_t \bigoplus_{[P']} \mathrm{Hom}_{\Lambda^T} (\Lambda^T e_{[P']},  \Lambda^T e_{P})_t\\
&=&   \bigoplus_t \bigoplus_{v} \bigoplus_{P_j} \mathrm{Hom}_{\mathcal C} (\Lambda e_{\mathcal N^v P'},  \Lambda e_{P})_t \\
&=&  \bigoplus_t \bigoplus_{v} \bigoplus_{P_j} D \mathrm{Hom}_{\mathcal C} ( e_{\mathcal N P}\Lambda,  e_{\mathcal N^v P'}\Lambda )_{l-t} \\
&\simeq &  D \bigoplus_t \bigoplus_{v} \bigoplus_{P_j} \mathrm{Hom}_{\mathcal C} ( e_{\mathcal N P}\Lambda, e_{\mathcal N^v P'}\Lambda  )_{l-t} \\
&\simeq &  D \bigoplus_t \bigoplus_{v} \bigoplus_{P_j} \mathrm{Hom}_{\mathcal C} ( e_{\mathcal N^{v+1} P}\Lambda, e_{ P'}\Lambda  )_{l-t} \\
&= &  D \bigoplus_t \bigoplus_{v} \bigoplus_{P_j} \mathrm{Hom}_{\mathcal D} ( e_{[\mathcal N P]}\Lambda^T, e_{[ P']}\Lambda^T  )_{l-t} \\
&=&  D e_{[\mathcal N P]}\Lambda^T, }$$
Thus $\Lambda^T e_{[P]}$ is injective, so $\Lambda^T$ is self-injective.
It also follows from this that the Nakayama permutation on the quiver of $\Lambda^T$ is identity.
\end{proof}

We see  that $\Lambda^T $ is a weakly symmetric algebra, and call it {\em the weakly symmetric algebra} of $\Lambda$.
Now we assume that the bound quiver $\Lambda$ is  $Q=(Q_0, Q_1, \rho)$.
%If $G$ acts freely on $\mathcal C$, we say $G$ acts freely on $\Lambda$.
Assume that $G$ act freely on $\Lambda$, that is,  $G$ acts freely on $\mathcal C$.
Then $\Lambda$ is a regular covering of $\Lambda^T$ with group $G$.
Let $Q^T $ be the quiver of $\Lambda^T$, then its vertex set is the set of the orbits of vertices of $Q_0$ under the Nakayama translation $\tau$ as the vertex set.
It follows from Lemma \ref{bil} that $dim_k\,  \Lambda_{1} e_i = dim_k\, \Lambda_{1} e_{\tau^r i}, $ and  $dim_k\, e_{ j} \Lambda_{1}  = dim_k\,e_{\tau^r j} \Lambda_{1} , $.
So the number of arrows starting or ending at each vertex of $Q$ and its images in $Q^T$
are the same.
We have the following description of unbounded quiver by the above theorem.

\begin{cor}\label{NKcov}
Assume that $G$ act freely on $\Lambda$.  As quivers, $Q$ is a regular covering of $Q^T$ with the group $G$ generated by the Nakayama translation.
\end{cor}

%In this case, one can also construct the universal covering $\widetilde{\Lambda}$ of $\Lambda$ (see \cite{gre}), whose bound quiver is denoted by $\widetilde{Q} = (\widetilde{Q_0}, \widetilde{Q_1}, \widetilde{\rho})$.
%This gives us a locally bounded algebra which is also the universal covering of $\Lambda^T$.

\medskip

%removed July 27

\section{Smash Product $\Lambda \# k \mathbb Z^*$, Separated Directed Quiver and Special Truncated Quiver}

Covering theory is very important in representation theory, especially in the study of self-injective algebra \cite{hwu,sy,oty}.
In this section, we study a  universal covering for a graded self-injective algebra $\Lambda$, its smash product with the infinite cyclic group $\mathbb Z$.
We describe the bound quiver of the smash product and study its properties.

Starting with a stable bound quiver $Q=(Q_0,Q_1,\rho)$ of Loewy length $l+1$, say, of $\Lambda$, we construct a directed quiver $(\overline{Q}_0,\overline{Q}_1) $ as follow.
\begin{itemize}
\item[]Vertex set: $$\overline{Q}_0=\{(i,n) | n \in \mathbb Z,  i \in Q_0\};$$
\item[]Arrow set: $$\overline{Q}_1 = \{(\alpha, n):(i,n) \to (j,n+1)| n \in \mathbb Z, \alpha: i \to j \in Q_1\}.$$
\end{itemize}
If $p = \alpha_s \cdots \alpha_1 $ is a path in $Q$, define $p[n] = (\alpha_s,n+s-1) \cdots (\alpha_1,n)$ for each $n \in \mathbb Z$.
Define relations
$$\overline{\rho} =\{\zeta[n]| \zeta \in \rho, n \in \mathbb Z \}$$
here $\zeta[n] = \sum_{t} a_t p_t[n]$ for each $\zeta = \sum_{t} a_t p_t \in \rho $. $\overline{Q}= (\overline{Q}_0,\overline{Q}_1, \overline{\rho})$ is a locally finite bound quiver. We call $\overline{Q}$ the {\em  separated directed quiver} of the stable bound quiver $Q$. We show that this bound quiver gives exactly the smash product of  $\Lambda$ with an infinite cyclic group $\mathbb Z$.

Recall that the smash product $\Lambda \# k \mathbb Z^*$ of $\Lambda$ with $\mathbb Z$ is the free $\Lambda$ module with basis $\mathbb Z^* = \{\delta_n | n \in \mathbb Z\}$, with multiplication defined by
$$x\delta_ny\delta_m = xy_{n-m}\delta_m $$
for $x,y \in \Lambda$, $y=\sum_{t=0}^l y_t$ with $y_t \in \Lambda_t$.
This is an infinite dimensional algebra without the unit.

Since $\delta_n$ centralize $\Lambda_0$, if $\{e_i| i \in Q_0\}$ is a complete sets of orthogonal idempotents of $\Lambda$, $\{e_i\delta_n| i \in Q_0, n \in\mathbb Z\}$ is a complete sets of orthogonal idempotents of $\Lambda \# k \mathbb Z^*$.

Assume that $0\neq x\in e_j \Lambda_t e_i$,  $0 \neq y \in e_{j'} \Lambda_{t'} e_{i'}$ are homogeneous elements of degree $t$ and $t'$, respectively, then $ y \delta_{m} x \delta_n = y x  \delta_n\neq 0$ if and only if $j=i', 0\neq yx$ and $m=n+t$.
Especially, if $0\neq x\in e_j \Lambda_1 e_i$ is homogeneous element of degree $1$ in $\Lambda$, then $e_{j'}\delta_m x e_{i'}\delta_n = e_{j'} x e_{i'}\delta_n\neq 0$ if and only if
$i=i', j=j', n=m'$ and $m=n+1$.
We see that $\Lambda \# k \mathbb Z^*$ is a locally finite-dimensional algebra, and $(\Lambda_1+ \cdots +\Lambda_l)k \mathbb Z^*$ is a nil ideal and $\Lambda \# k \mathbb Z^*/(\Lambda_1+ \cdots +\Lambda_l)k \mathbb Z^*$ is semi-simple.
Write $e_i\delta_n$ as $e_{(i,n)}$, and for a homogeneous element $x \in \Lambda_t$, write $x\delta_n$ as $x[n]$.
We see that for $\alpha:i \to j \in Q_1$, $\alpha \delta_n = \alpha[n] :(i,n) \to (j,n+1)$ is an arrow in the quiver of $\Lambda \# k \mathbb Z^*$.
This shows that the Gabriel quiver of $\Lambda \# k \mathbb Z^*$ is exactly $(\overline{Q}_0, \overline{Q}_1 )$.

%\begin{prop}\label{smashZ} \end{prop}

Clearly, if $\alpha_l\cdots \alpha_1 $ is a path in quiver $Q$, then $\alpha_l[n+l-1]\cdots \alpha_2[n+1]  \alpha_1[n]= \alpha_l\cdots \alpha_2\alpha_1[n] $ is a path in quiver $\overline{Q}$.
Also, we have that $\sum_s a_s p_s = 0 $ for paths  $p_s$ of $Q$ and $a_s \in k$ if and only if for all $n$, we have $\sum_s a_s p_s[n] = 0 $ in $\Lambda \# k \mathbb Z^*$.
Thus, a path $\alpha_l\cdots \alpha_1 $ is a maximal bound path of the bound quiver $Q$ if and only if for all $n$, $\alpha_l[n+l-1] \cdots \alpha_2[n+1]  \alpha_1[n]= \alpha_l\cdots \alpha_2\alpha_1[n] $ is a maximal bound path of  $\Lambda \# k \mathbb Z^*$.
So we see the maximal path in  $\Lambda \# k \mathbb Z^*$ starting at $(\tau i,n-l)$ ends at $(i,n)$.
Especially, maximal paths starting at the same vertex end at the same vertex, and they are all $k$-linear dependent.
Define $\overline{\tau} (i,n) = (\tau i, n-{l})$, this induces an automorphism on the bound quiver $\overline{Q}$.
This shows that $\Lambda \# k \mathbb Z^*$ is given by the relations $\overline{\rho} =\{\zeta[n]| \zeta \in \rho, n \in \mathbb Z \}$.
And $\overline{Q}= (\overline{Q}_0,\overline{Q}_1, \overline{\rho})$ is a stable quiver with the Nakayama translation $\overline{\tau}$.
So we have the following theorem.

\begin{thm}\label{lsquiver}
\begin{enumerate}
\item $\Lambda \# k \mathbb Z^*$ is a self-injective algebra of Loewy length $l+1$ with the bound quiver $\overline{Q} = (\overline{Q}_0, \overline{Q}_1, \overline{\rho})$.

\item The Nakayama translation $\overline{\tau}$ of  on $\overline{Q}$ is defined by $\overline{\tau} (i,n) = (\tau i, n-{l})$.

\item $\overline{Q}$ is a locally finite stable bound quiver of Loewy length $l+1$.
\end{enumerate}
\end{thm}

We will write $\tau$ for $\overline{\tau}$ when no confusion appears.

The quiver of $\overline{Q}$ is different from the usual quiver $\mathbb Z Q$ used in representation theory of algebras, and it is usually not connected.
Assume that the lengths of minimal oriented cycles in $Q$ are $l_1,\ldots, l_r$,  and let $d = \mathrm{gcd} (l_1,\ldots,l_r)$ be their greatest common divisor.
The number of connected components of $\overline{Q}$ is given below.

\begin{prop}\label{conn} Let $Q$ be a finite connected stable quiver.
Then $\overline{Q}$ has $d$ connected components.

If $Q$ contains a loop, then $\overline{Q}$ is connected.
 \end{prop}

The proposition follows easily from the following lemma.

\begin{lemma}\label{inonecom}
For any vertex $j$ of $Q$, $(j,m')$ and $(j,m'')$ are in the same connected component of $\overline{Q}$ if and only if $m'-m''\equiv 0\bmod d$.
\end{lemma}
\begin{proof}
If both  $(j,m')$ and $(j,m'')$ are in the same connected component, then by Lemma ~\ref{path}, there is some vertex $(i,n)$ such that there is a path from $(i,n)$ to  $(j,m')$ and a path from $(i,n)$ to $(j,m'')$, and there is a path from $(j,m)$ to $(i,n)$ for some integer $n, m$.
Clearly, $m < n < \mathrm{min\,}\{m', m''\}$. We see that the paths from $(j,m)$ to $(j,m')$ and $(j,m'')$ are got from oriented cycles of $Q$, so $d| m'-m$ and  $d| m''-m$, so $d|m'-m''$ and $m'-m''\equiv 0\bmod d$.

Now assume that $m'- m'' \equiv 0\bmod d$, we may assume that $m'-m''=\sum_{t=1}^r s_tl_t$ for integers $ s_t, l_t \in \mathbb Z$. We can choose a vertex $i_t$ from each minimal oriented cycle $q_t$ and a path $p_t$ from $j$ to $i_t$.
The walk $$p_r^{-1}q_r^{s_r}p_r\cdots p_2^{-1}q_2^{s_2}p_2p_1^{-1}q_1^{s_1}p_1$$ in $Q$ gives rise to a walk in $\overline{Q}$ from $(j,m'')$ to $(j,m')$,  so $(j,m')$ and $(j,m'')$ are in the same connected component of $\overline{Q}$.
\end{proof}

Take a vertex $i \in Q_0$, denote by $(\overline{Q},i)$ the connected component of $\overline{Q}$ containing the vertex $(i,0)$.
If there exist paths $p,q$ of the same length  from $i$ and $i'$, respectively, to the same ending vertex $j$, then $(\overline{Q},i) = (\overline{Q},i')$. If there is an arrow $\alpha: i_0 \to j_0$ in $Q$, then we have an isomorphism from $(\overline{Q},i_0)$ to $(\overline{Q},j_0)$ sending $(i'',n)$ to $(i'', n-1)$.
Since $Q$ is connected, we have the following proposition.

\begin{prop}\label{cntiso}
All the connected components of $\overline{Q}$ are isomorphic.
\end{prop}

By definition, one sees easily that the following holds.
\begin{prop}\label{nocyc}
The quiver $\overline{Q}$ contains no oriented cycle.
\end{prop}

\medskip

Fix a connected component $(\overline{Q},i_0)$, we now study its truncation  with respect to the Nakayama translation.

The quiver obtained by taking the vertex set  $$(Q^N,i_0)_0 =\{(j,n ) \in (\overline{Q},i_0)| 0 \le n \le l-1\}$$ together with all the arrows in $(\overline{Q},i_0)$ among these vertices is called a {\em special truncated quiver} of $\overline{Q}$ with $i_0$ as a source and is denoted by $(Q^N, i_0)$.

Note that for all $t$, we have that $(Q^N, i_0) \simeq (Q^N, \tau^t i_0) $. Let $d$ be the greatest common divisor of the lengths of the oriented cycles of $Q$.
Take a path $i_1 \to i_2 \to \ldots \to i_{d}$ of length $d-1$ in $Q$, then for any vertex $j$ of $Q$, we have that $(Q^N, j) \simeq (Q^N, i_t)$ for some $1 \le t \le d$.
So the union $\cup_{t} (Q^N, i_t)$ is independent of the choice of the path.
The full subquiver $Q^N$ of $\overline{Q}$ with the vertex set $$Q^N_{0}=\{(i,n)| i\in Q_0, 0 \le n \le l-1\}$$ is called {\em the total special truncated quiver} of $Q$.
We use the same notations $(Q^N, i_0)$ and $Q^N$ for the bound quivers with induced relations.
Clearly the following proposition holds.

\begin{prop}\label{sptrqv}
$Q^N$ is isomorphic to a disjoint union of  special truncated quivers $$Q^N \simeq  \cup_{1\le t\le d}  (Q^N, i_t),$$
for any path $i_1 \to i_2 \to \ldots \to i_{d}$ of length $d-1$ in $Q$.
\end{prop}
The following is obvious.

\begin{prop}\label{same}
For any vertex $(j,0)$ of $(\overline{Q}, i_0)$, we have $(\overline{Q},i_0)=(\overline{Q},j)$ and $(Q^N, i_0) = (Q^N,j)$.
\end{prop}

Write $\mathcal{P} = \mathcal{P} (\Lambda \# k \mathbb Z^*)$ be the category of finitely generated indecomposable projective $\Lambda \# k \mathbb Z^*$-modules.
Let $P_{i,n} = \Lambda \# k \mathbb Z^* e_{(i,n)}$ be the indecomposable  projective $\Lambda \# k \mathbb Z^*$-modules corresponding to the vertex $(i,n)$.
Let $\Lambda(Q^N, i_0) =\mathrm{End}_{\mathcal{P} } \bigoplus_{(i,n) \in (Q^N, i_0)_0} P_{i,n} $.
$\Lambda(Q^N, i_0) $ is given by the bound quiver $ (Q^N, i_0)$.
Clearly, for all $t$, $\Lambda(Q^N, i_0) \simeq \Lambda(Q^N, \tau^t i_0) $.
An algebra $\Lambda(Q^N, i_0)$ is called a {\em special truncated algebra}.
If $Q$ is the bound quiver of a graded self-injective algebra $\Lambda$, we also call this algebra a {\em special truncated algebra of $\Lambda$}.

Fix a path $i_1 \to i_2 \to \ldots \to i_{d}$ of length $d-1$ in $Q$, write $$\Lambda^N =   \bigoplus_{1\le t\le d}  \Lambda(Q^N, i_t),$$
then up to isomorphism, $\Lambda^N $ is independent of the choice of the path, and it is given by the bound quiver $Q^N$.
We call this algebra {\em the total special truncated algebra}.

\section{The Beilinson Algebra, its Trivial Extension and Repetitive algebra}

In \cite{cx}, Chen introduces the Beilinson algebra and shows that the category of graded modules of a well graded self-injective algebra is equivalent to the category of the graded modules of the trivial extension of its Beilinson algebra.
Now we describe the Beilinson algebra of a graded self-injective algebra, its trivial extension and its repetitive algebra, using algebras and the bound quivers induced in the last section.

Let $\Lambda= \Lambda_0 + \Lambda_1 +\cdots + \Lambda_l$ be a basic graded self-injective algebra. The Beilinson algebra of $\Lambda$, defined in \cite{cx},
is the algebra the form $$b(\Lambda) = \matr{ccccc}{\Lambda_0 &\Lambda_1&\cdots  &\Lambda_{l-2} &\Lambda_{l-1} \\
0&\Lambda_0 &\cdots  &\Lambda_{l-3} &\Lambda_{l-2} \\ \vdots  &\vdots  & \ddots &\vdots  & \vdots \\
0 &0 & \cdots &\Lambda_0  & \Lambda_1\\
0& 0& \cdots& 0&\Lambda_0   \\
}. $$
We have the following theorem.

\begin{thm}\label{bei}
Let $Q$ be the bound quiver of a basic graded self-injective algebra $\Lambda$.
Then $Q^N$ is the bound quiver of its Beilinson algebra.
\end{thm}

\begin{proof} Since $\Lambda$ is naturally graded by the lengths of paths, $\Lambda_0$ is a vector spaces with the primitive idempotents (trivial paths) as its basis, and $\Lambda_t$ has a basis consisting of paths of
length $t$.

So we have that $$b(\Lambda) =b(\Lambda)_0 + b(\Lambda)_1 + \cdots +b(\Lambda)_{l-1}, $$ with
$$b(\Lambda)_t = \matr{cccccc}{0&\cdots &0&\Lambda_t &\cdots   &0\\
\cdot  &\cdots &\cdot &\cdot  & \cdots   & \cdot \\
0& \cdots &0  &0&\cdots   &\Lambda_t\\
0& \cdots & 0&0 &\cdots     &0 \\
\cdot  &\cdots &\cdot &\cdot  &  \cdots   & \cdot \\
0& \cdots & 0&0 &\cdots  &0   \\
}, $$
for $t=0,1\ldots,l-1$.
The Jacobson radical $\mathbf r b(\Lambda) = b(\Lambda)_1 + \cdots +b(\Lambda)_{l-1}$,   $ b(\Lambda)/\mathbf r b(\Lambda) \simeq \matr{cccc}{\Lambda_0 &0&\cdots  &0 \\
0&\Lambda_0 &\cdots  &0 \\
\vdots  & \vdots & \ddots &\vdots  \\
0&0&\cdots  &\Lambda_0\\
}$,  $\mathbf r b(\Lambda)/\mathbf r^2 b(\Lambda) \simeq \matr{cccc}{0&\Lambda_1 &\cdots  &0 \\
\vdots  & \vdots & \ddots &\vdots  \\
0&0&\cdots  &\Lambda_1\\
0&0&\cdots  &0\\
}.$
Denote by $e_{(i,n)}$ the matrix with a single idempotent $e_i$ of $\Lambda$ at the $(l+1-n,l+1-n)$ position, for $n=1, \ldots, l$, and denote by  $(\alpha,n)=\alpha[n]$ the matrix with a single arrow $\alpha$ at $(l+1-(n+1),l+1-n)$ position, for $n=1, \ldots, l-1$.
We see easily that the quiver of $b(\Lambda)$ is exactly the total special truncated quiver $Q^N$ of $Q$.
Regard the idempotents as trivial paths, denoted by $p[n]$ the matrix of a single path $p$ of length $t$ at position $(l+1-n+t, l+1-n)$, this is a path of length $t$ in $b(\Lambda)_t$.
So $p[n]= (\alpha_t,n+t-1)\cdots(\alpha_1,n)$ if $p=\alpha_t\cdots \alpha_1$.
Clearly $ \sum_{t} a_t p_t[n] =0 $ in $b(\Lambda)$ if and only if $ \sum_{t} a_t p_t=0$ in $\Lambda$.
So the relations of $b(\Lambda)$ are identified with those of $Q^N$, and this shows that the bound quiver of $b(\Lambda)$ is exactly the same as the total special truncated quiver $Q^N$. \end{proof}

This also shows that $b(\Lambda) \simeq \Lambda^N$.

\medskip

Consider the category $\mathrm{ind}\mathcal{P}= \mathrm{ind }\mathcal{P} (\Lambda \# k \mathbb Z^*)$.
Let $G$ be group generated by the Nakayama functor  $\mathcal N$, $G$ acts freely on the objects of $\mathrm{ind} \mathcal{P} $.
For any positive integer $r$, let $G(r) = (\mathcal N^r)$ be the subgroup of $G$ generated by $\mathcal N^r$, $G_r$ acts freely on the objects of $\mathrm{ind} \mathcal{P} $, too.
Let $\mathcal{P}$ be a finite generated basic $G_r$-orbit generator.
Let  $\Lambda^{T,r} = O(G_r , P_r)$ be the orbit algebra.
Similar to Theorem ~\ref{covA}, we see that $\Lambda^{T,r}$ is a graded self-injective algebra of Loewy length $l+1$ and $\Lambda \# k \mathbb Z^*$ is a Galois covering of $\Lambda^{T,r}$ with the group $G_r$.
Let $\Lambda^T= \Lambda^{T,1}$ and $Q^T=Q^{T,(r+1)}$, then $\Lambda^T = k(Q^T)$.
By Theorem ~\ref{covA}, we have

\begin{prop}\label{gcv1}
$\Lambda^{ T,r}$ is a regular covering of $\Lambda^T$ with the group $\mathbb Z / r\mathbb Z$.
\end{prop}

Clearly $\Lambda^{T,r}$ is an intermediate covering of $\Lambda^T$ for $r > 1$.
Let $Q^{T,r}$ be the bound quiver of $\Lambda^{T,r}$.
Then the vertex set of $Q^{T,r}$ is the set of the orbits of the vertices $\overline{Q}$ with respect to the group generated by the $r$th power of Nakayama translation.
It follows from Corollary \ref{arrows} that the number of arrows from $i$ to $j$ is the same as the number of the arrows from $\tau i$ to $\tau j$, so we may extend $\tau$ to a bijective map on the arrows, and $\tau$ is extended to a quiver automorphism of $\overline{Q}$.
The arrows of $Q^{T,r}$ are regarded as the orbits of the arrows under $\tau^r$.
So the quiver of $Q^{T,r}$ is the orbit quiver of $\overline{Q}$, obtained from the separated directed quiver $\overline{Q}$ by identifying the vertices and arrows in a $\tau^r$-orbit, respectively.
So the vertices $(j,m)$ and $(\tau^{-rt} j, m+rtl)$ of $\overline{Q}$ are identified in $Q^{T,r}$ for all integers $t$.

We now turn to the $r=1$ case.
The following proposition follows easily from Corollary ~\ref{arrows}.

\begin{prop}\label{orbitqv}
Each connected component of the quiver ${Q}^T$ is obtained from a connected component of a special truncated quiver by adding an arrow from $(j,l)$ to $(i,0)$ for each independent bound path from $(i,0)$ to $(j,l)$.
\end{prop}

Let $\Lambda$ be an algebra and $M$ be a $\Lambda$-bimodule. Recall the trivial extension $\Lambda\ltimes M$ of $\Lambda$ by $M$ is the algebra defined on the vector spaces $\Lambda\oplus M$ with the multiplication defined by $$(a,x)(b,y)=(ab, ay+xb)$$ for $a,b\in \Lambda$ and $x,y\in M$.

Trivial extension of $\Lambda$ by its dual $D\Lambda$ is called the trivial extension of $\Lambda$ and we denote it as $\iota (\Lambda) = \Lambda \ltimes D \Lambda$.

Let $\sigma$ be an automorphism of $\Lambda$. Let $M$ be a $\Lambda$-bimodule.
Define the twist $M^{\sigma}$ of $M$ as the bimodule with $M$ as the vector space.
The left multiplication is the same as $M$, and
the right multiplication is twisted by $\sigma$, that is,  defined by $xb=x\sigma(b)$ for all $x \in M^{\sigma}$ and $b \in \Lambda$.
Define the twisted trivial extension $\iota_{\sigma} (\Lambda) = \Lambda \ltimes D \Lambda^{\sigma}$ to be the trivial extension of $\Lambda$ by the twisted $\Lambda$-bimodule $D \Lambda^{\sigma}$.

\medskip
Let $Q$ be a stable quiver of Loewy length $l+1$.
and let $\overline{Q}$ be its separated directed quiver.
Let $Q'=(Q^N, i_0)$ be a special truncated quiver of $\overline{Q}$, and let  $Q^T(i_0)$ be a connect component of $Q^T$ containing $(i_0,0)$.
Let $\Lambda' = k(Q')$ and $\Lambda^{T'} = k(Q^T(i_0))$.

\begin{lemma}\label{embed}
An algebra is a subalgebra of its trivial extension algebra.

$\Lambda'$ is a subalgebra of $\Lambda^{T'}$.
\end{lemma}

\begin{proof} The first assertion follows directly.

For the second assertion, we observe that $Q^T(i_0) $ is obtained from $Q'$ by adding certain arrows from vertices $(i,l-1)$ to $(j,0)$ and relations with paths containing the new arrows.
Since the relations added does not concern any paths in $Q'$,  $\Lambda' =k(Q')$ is a subalgebra of $k(Q^T(i_0))$.
\end{proof}
Regard $Q'$ as a bound subquiver of $Q^T(i_0)$ with the same vertex set, index the vertices as $(j,m)$ with $0 \le m \le l-1$.
$\Lambda^{T'}$ is a graded self-injective algebra  with Loewy length $l+1$.
Write $\Lambda^{T'} = \Lambda^{T'}_0 +\Lambda^{T'}_1 + \cdots + \Lambda^{T'}_{l-1} + \Lambda^{T'}_l$, where $\Lambda^{T'}_t$ is the homogeneous component of degree $t$.

Let $\Lambda' = \Lambda'_0 +\Lambda_1' + \cdots + \Lambda'_{l-1}$, where $\Lambda'_t$ is the homogeneous component of degree $t$.
Let $\iota (\Lambda')$ be the trivial extension of $\Lambda'$.
By Lemma ~\ref{embed}, $\Lambda'$ is subalgebra both of $\iota (\Lambda')$ and of $\Lambda^{T'}$.
Clearly, $\Lambda'_0\simeq \iota (\Lambda')_0\simeq  \Lambda^{T'}_0$ is semisimple algebra, and we identify them with $\Lambda'_0$.
Write $e_{(i,t)}$ for the primitive idempotent corresponding to the vertex $(i,t)$.

For any vertices $(i,t_1), (j,t_2) \in Q^T(i_0)_0$ with $0 \le t_1 < t_2 \le l-1$, we have
$$e_{(j,t_2)} \iota (\Lambda')_{t_2-t_1} e_{(i,t_1)}  =e_{(j,t_2)} \Lambda'_{t_2-t_1}e_{(i,t_1)}=e_{(j,t_2)} \Lambda^{T'}_{t_2-t_1} e_{(i,t_1)}.$$

The vertex set of $Q^T(i_0)$ is the same as that of $Q'$, decompose $\Lambda^{T'}$ accordingly,
$$\arr{ccl}{\Lambda^{T'} &=&\bigoplus_{(i,t_1),(j,t_2)\in Q^T(i_0)_0, t_1 \le t_2} e_{(j,t_2)}\Lambda^{T'}_{t_2-t_1}e_{(i,t_1)} \\&& \oplus  \bigoplus_{(i,t_1),(j,t_2)\in Q^T(i_0)_0, t_1 \le t_2} e_{(i,t_1)}\Lambda^{T'}_{l+t_1-t_2} e_{(j,t_1)} \\&=& \Lambda'\oplus  \bigoplus_{(i,t_1),(j,t_2)\in (Q^N,i_0)_0, t_1 \le t_2} e_{(i,t_1)}\Lambda^{T'}_{l+t_1-t_2} e_{(j,t_1)}.}$$

Let $$M=\bigoplus_{(i,t_1),(j,t_2)\in (Q^N,i_0)_0, t_1 \le t_2} e_{(i,t_1)}\Lambda^{T'}_{l+t_1-t_2} e_{(j,t_1)}.$$

\begin{lemma}\label{bimoduleM}
$M$  is a $\Lambda'$-bimodule and $\Lambda^{T'}$ is a trivial extension of $\Lambda'$ by $M$.
\end{lemma}
\begin{proof}
Clearly, $M$  is a $\Lambda'$-bimodule.

Since for any $i,j$,  each bound path in $Q^T(i_0)$ passes through the arrow from $(i,l-1)$ to $(j,0)$ at most once and each path in $M$ passes through arrow from $(i,l-1)$ to $(j,0)$ at least once.
So for any elements $x,y\in M $, we have $xy=0$ in $\Lambda^{T'}$.
Hence $\Lambda^{T'}$ is a trivial extension of $\Lambda'$ by $M$.
\end{proof}

\begin{lemma}\label{twist}
$M$  is isomorphic to $D{\Lambda'}^{\sigma}$ for some automorphism $\sigma$ of $\Lambda '$.
\end{lemma}
\begin{proof}
By Theorem ~\ref{covA}, $\Lambda^{T'}$ is self-injective.
From its proof we see that all the projectives of $\Lambda^{T'}$ have the same Loewy length as those of $\Lambda\# k G^*$, which is, the same as the Loewy length of the projectives of $\Lambda$.
So $\Lambda^{T'}$ is well graded in the sense of \cite{cx}, and the Lemma follows from Lemma 2.5 of \cite{cx}.
\end{proof}

So we get the following theorem immediately.

\begin{thm}\label{TE}
Let $\Lambda'$ be the algebra given by the bound quiver $Q'=(Q^N, i_0)$ and let $\iota (\Lambda')$ be the trivial extension of $\Lambda'$. Let $\Lambda^{T'}$ be the orbit algebra of a connected component of of $\Lambda \# k \mathbb Z^*$ containing $e_{(i_0,0)}$ with respect to the Nakayama functor. Then there is an automorphism $\sigma$ of $\Lambda'$ such that $$\Lambda^{T'} \simeq \iota_{\sigma} (\Lambda').$$
\end{thm}

As a corollary, we have that $\Lambda^{T} \simeq \iota_{\sigma} (\Lambda^N)$.

\medskip

Note that $I = \{(i,n)|i \in Q_0, n \in \mathbb Z\}$ is the index set of a complete set of orthogonal idempotents in   $ \Lambda \# k \mathbb Z^*$.
Set $I[s] = \{(i,n) |i \in Q_0, (s-1)l \le n < sl  \}$ for $s \in \mathbb Z$.
Let $e_{I[s]} = \sum_{(i,n)  \in I[s] } e_{(i,n)} $, and let $$M[s,t] = e_{I[s]}  \Lambda \# k \mathbb Z^* e_{I[t]} = \sum_{(i,n)  \in I[s]}\sum_{(j,m)  \in I[t] } e_{(i,n)} \Lambda \# k \mathbb Z^*  e_{(j,m)}$$
for all $s,t \in \mathbb Z$.
By Theorem ~\ref{lsquiver}, maximal bound path in  $ \Lambda \# k \mathbb Z^*$ has length $l$, and all paths of length larger than $l$ are zero.
Since $ e_{(j,n)} \Lambda \# k \mathbb Z^*  e_{(i,m)}$ is spanned by paths of length $n-m$, thus  $ e_{(j,n)} \Lambda \# k \mathbb Z^*  e_{(i,m)}=0$ if $n-m < 0$ or $n-m > l$.
This leads to the following lemma.

\begin{lemma}\label{Mst}
$ \Lambda \# k \mathbb Z^* = \bigoplus_{s,t \in \mathbb Z} M[s,t]$ as vector spaces.

$M[s,t]  = 0$ if $s \neq t, t+1$.

$M[s,t] M[s,t] = 0$ for all $s\neq t$.
\end{lemma}

Let $\Gamma[s] = M[s,s] = e_{I[s]}  \Lambda \# k \mathbb Z^* e_{I[s]}$, it is an algebra with the unit $e_{I[s]}$.
It follows easily that

\begin{lemma}\label{nearrep}
$$\Lambda \# k \mathbb Z^*  \simeq \left(\arr{ccccccc}{\ddots \\ & \Gamma[s+1] & M[s+1,s] & \\ & &\Gamma[s] & M[s,s-1]  \\ & & &\Gamma[s-1] &\\& & &&\ddots}
\right).$$
\end{lemma}

Let ${Q}[s]$ be the bound quiver with vertex set $I[s]$ and induced relations.
Then ${Q}[s]$ is a shift of $Q[0]=Q^N$ and  $\Gamma[s] = k({Q}[s])$.
The shift induces isomorphisms
$$\psi_{s,t}: \Gamma[s] \to \Gamma[t] $$
between these algebras, and $M[s,t]$ is a $\Gamma[s]$-$\Gamma[t]$-bimodule.
Using Lemma \ref{bil}, similar to the argument in the proof of Theorem \ref{covA}, we have

\begin{lemma}\label{Mbim}
$M[s+1,s] \simeq D\Gamma[s]$ as right $ \Gamma[s] $-module and $M[s+1,s] \simeq D\Gamma[s+1]$ as left $ \Gamma[s+1] $-module.
\end{lemma}

Now identify $\Gamma[s]$ with $\Gamma[0] = \Lambda^N$ for all $s$, using the above isomorphisms, then $M[s+1,s]$ are identified with $D\Lambda^N$, for all $s$.
The following theorem follows from Lemma ~\ref{nearrep}.

\begin{thm}\label{repetitive}
$\Lambda \# k \mathbb Z^* \simeq \widehat{\Lambda^N}$ is the repetitive algebra of $\Lambda^N$.
\end{thm}

Since $\Lambda^N = \bigoplus_{i}(\Lambda^N,i)$, where $i$ runs over a path of length $d-1$.
So we have $\Lambda \# k \mathbb Z^* = (\bigoplus_{i}(\Lambda^N,i))\sphat =\bigoplus_{i}\widehat{(\Lambda^N,i)}$.
Obviously, $\widehat{(\Lambda^N,i)}$ is an algebra with bound quiver $(\overline{Q},i)$ for each $i$, so these direct summands of $\Lambda \# k \mathbb Z^*$ are isomorphic.
Since $\overline{Q}$ has $d$ connected components, so we have that.

\begin{prop}\label{directsum}
$\Lambda \# k \mathbb Z^* =\bigoplus_{i}\widehat{(\Lambda^N,i)}$
is a direct sum of $d$ isomorphic algebras.
\end{prop}

It follows from Lemma 2.1 of \cite{cx} that $\Lambda, \Lambda^T$ and $\iota (\Lambda')$ have equivalent categories of graded modules. We have the following theorem from Lemma ~\ref{nearrep}.

\begin{thm}\label{bgg} Let $\Lambda$ be a finite-dimensional graded self-injective algebra  over $k$.
Then we have equivalences among the following triangulated categories:

\begin{enumerate}
\item $\mathcal D^b(\Lambda^N)$, the bounded derived category
of the category $\bmod\, \Lambda^N$ of finite  generated $\Lambda^N$-modules.

\item $\underline{\mathrm{gr\,}}{\iota(\Lambda^N)}$, the stable category of finite generated graded $\iota (\Lambda^N)$-modules.

\item $\underline{\mathrm{gr\,}}{ \Lambda}$, the stable category of finite generated graded $\Lambda$-modules.

\item $\underline{\bmod}\,\Lambda \# k \mathbb Z^*$, the stable category of finite generated $\Lambda \# k \mathbb Z^*$-modules.
\end{enumerate}
 \end{thm}

\begin{proof}
Since in our case,  the degree $0$ part of the algebra $\Lambda$, $\Lambda_0$ is semisimple and hence of finite global dimension, the equivalence of the first three categories follows from Lemma 2.1 and Theorem 1.1 of \cite{cx}. The equivalence of the first and the last categories follows from Theorem \ref{repetitive} and Theorem 4.9 of \cite{h1}
\end{proof}

\section{$\tau$-slices, $\tau$-slice Algebras and $\tau$-mutations}

In the last section, we discuss the Beilinson algebra defined in \cite{cx} using bound quiver and Nakayama translation.
Beilinson algebra is a solution to the problem of finding algebras of finite global dimension whose derived categories are equivalent to the stable category of a given graded self-injective algebra.
Now we give a systematical way of finding algebras with this property, and investigate the interrelation among these algebras.

Let $\overline{Q}$ be a directed stable quiver of Loewy length $l+1$ with Nakayama translation $\tau$, let $v \in Q_0$.
We define the {\em $\tau$-hammock starting at $v$} as the full subquiver with the vertex set $$H^v=\{ u \in Q_0 | \mbox{there is a bound path  from $v$ to $u$}\}. $$
Dually we define the {\em $\tau$-hammock $H_v$ ending at $v$}.
$H^{ v}$ is the support of the projective cover of the simple corresponding to the vertex $v$ and $H_v$ is the support of the injective envelope of the simple  corresponding to the vertex $v$.
The $\tau$-hammock starting at $\tau v$ coincides with the $\tau$-hammock ending at $v$, that is
$$H^{\tau v} =H_{v}.$$

Let  $Q$ be a finite stable bound quiver, and let $(\overline{Q},i_0)$ be a separated directed quiver of  $Q$.
Let $Q'$ be a  full bound subquiver of $(\overline{Q},i_0)$.
$Q'$ is called a {\em $\tau$-slice} of $Q$ provided that for each vertex $v $ of $(\overline{Q},i_0)$, the intersection of the $\tau$-orbit of $v$ and the vertex set of $Q'$ is a single-point set.
A vertex $v$ of a $\tau$-slice $Q'$ is called {\em $\tau$-initial} provided that $Q' \cap H^v =H^v \setminus \{\tau^{-1} v\}$.
A vertex $v$ is called {\em $\tau$-terminal} provided that $Q' \cap H_v =H_v\setminus \{\tau
v\}$.
Thus we have that  $H^v\setminus \{\tau^{-1} v\}$ is a full bound subquiver of $Q'$ if $v$ is $\tau$-initial and $H_v\setminus \{\tau v\}$ is a full bound subquiver of $Q'$ if $v$ is $\tau$-terminal.

A $\tau$-slice $Q^S$ of a stable bound quiver $Q$ is called a {\em complete $\tau$-slice} if it satisfies the following condition.
\begin{enumerate}
\item Each source of $Q^S$ is $\tau$-initial;
\item Each sink of $Q^S$ is $\tau$-terminal;
\item Assume that $v \to u$ is an arrow of $(\overline{Q},i_0)$.
    If $v$ is a vertex of $Q^S$, then either $u$ or $\tau u$ is a vertex of $Q^S$;
    if $u$ is a vertex of $Q^S$, then either $v$ or $\tau^{-1} v$ is a vertex of $Q^S$.
\end{enumerate}

Let $Q^S$ be a complete $\tau$-slice in $(\overline{Q},i_0)$, and let $r$ be an integer.
The full bound quiver $Q^S(r)$ with the vertex set $$Q^S(r)_0 = \{(j,m-r)|(j,m)\in Q^S_0\}$$ is a complete $\tau$-slice in $(\overline{Q},i'_0)$ for some vertex $i'_0$ and it is isomorphic to $Q^S$.
We call $Q^S(r)$ {\em a shift of $Q^S$}.
So we see that up to shift,  a complete $\tau$-slice is independent of the choice of the component of $\overline{Q}$.
We regard a complete $\tau$-slice in $(\overline{Q},i'_0)$ as a bound quiver with the relations induced from the relations of $(\overline{Q},i'_0)$.
The algebra defined by a complete $\tau$-slice is called a {\em $\tau$-slice algebra} of $Q$.

If we start with a finite dimensional graded self-injective algebra $\Lambda$.
We can construct its separated directed quiver $\overline{Q_{\Lambda}}$ and consider the complete $\tau$-slices. A $\tau$-slice algebra obtained in this manner is also called a {\em  $\tau$-slice algebra of $\Lambda$.}

Since a separated directed quiver contains no oriented cycle, so does its subquiver.
Hence one gets the following proposition.

\begin{prop}\label{fgd} A $\tau$-slice algebra of $\Lambda$ is subalgebra of $\Lambda\# k \mathbb Z^*$ of finite global dimension. \end{prop}

Let $Q^S$ be a  complete $\tau$-slice in a separated directed quiver of  graded self-injective algebra $\Lambda$.
The number $d(Q^S) = max\{n |(i,n) \in Q^S\} - min\{n |(i,n) \in Q^S\} $ is called {\em the depth} of $Q^S$.
If $Q$ is finite, then $Q^S$ is finite and contains no oriented cycle, we may shift it in $\overline{Q}$ such that $min\{n |(i,n) \in Q^S\} =0$.

We have that $d(Q^S) \ge l-1$, and when $d(Q^S) =l- 1$, we call $Q^S$ an {\em initial $\tau$-slice.}
Clearly, we have the following proposition by shifting.

\begin{prop}\label{inialg}
$Q^S$ is an initial complete $\tau$-slice if and only if $Q^S \simeq (Q^N,i_0)$ for some vertex $i_0$.
\end{prop}

Let $Q^S$ be a complete $\tau$-slices in $(\overline{Q}, i_0)$ and let $(i,m)$ be a sink of $Q^S$.
Define the {\em $\tau$-mutation} $s_i^-(Q^S)$ of $Q^S$ at $i$ as the full bound subquiver in $(\overline{Q}, i_0)$ obtained by replacing the vertex $(i,m)$ by its Nakayama translation $(\tau i, m-l)$.
Dually, for a source $(j,m)$ of $S$, we define the {\em $\tau$-mutation} $ s_j^+(Q^S)$ of $Q^S$ at $j$ as the full bound subquiver of $(\overline{Q},i_0)$ obtained by replacing the vertex $(j,m)$ by its inverse Nakayama translation $(\tau^{-1} j, m+l)$.
Clearly,  a $\tau$-mutation of a complete $\tau$-slice in $(\overline{Q}, i_0)$ is again a complete $\tau$-slice in $(\overline{Q}, i_0)$.
If $(i,m)$ is a sink of $S$, then $s_i^+s_i^- Q^S = Q^S$, and if $(i,m)$ is a source of $Q^S$, then $s_i^-s_i^+ Q^S = Q^S$.

Let $Q^S$ be a complete $\tau$-slice and let $\sigma$ be a $\tau$-mutation defined on $Q^S$.
If $\Lambda'$ and $\Lambda''$ are $\tau$-slice algebras defined by $Q^S$ and $\sigma Q^S$, respectively, $\Lambda''$ is called a {\em  $\tau$-mutation of  $\Lambda'$}, and write it as $$\Lambda''=\sigma \Lambda'.$$

The following lemma follows easily from induction on the depth of the complete $\tau$-slices and on the number of pairs of vertices which reach the maximal depth.

\begin{lemma}\label{refl}
Let $Q$ be a finite stable bound quiver with Nakayama translation $\tau$, and let $Q^S$ be a complete $\tau$-slice of $Q$.
Then there is a sequence $\sigma_1,  \ldots, \sigma_r$ of $\tau$-mutations such that $\sigma_r \cdots \sigma_1 Q^S$ is an initial complete $\tau$-slice $(Q^N,i_0)$,  up to a shift.
\end{lemma}

The following theorem asserts that the trivial extensions of $\tau$-slice algebras are invariant under the $\tau$-mutations.

\begin{thm}\label{trexs} The trivial extensions of all the  $\tau$-slice algebras of  a finite stable bound quiver are isomorphic.
\end{thm}

\begin{proof}  Let $Q^S$ be a complete $\tau$-slice of a stable bound quiver $Q$.
Then by Lemma ~\ref{refl}, there is a sequence $\sigma_1,\ldots,\sigma_r$ of $\tau$-mutations such that $Q^S= \sigma_r  \sigma_{r-1} \cdots\sigma_1 (Q^N,v)$ for an initial complete $\tau$-slice $Q'=(Q^N,v)$.

Let $Q^t = \sigma_t \cdots\sigma_1 (Q^N,v)$ and let $\Lambda^t = k(Q^t)$ be the algebra given by the bound quiver $Q^t$.

Let $\Lambda' = k(Q')$, and let  $\Lambda^E=\iota (\Lambda') $ be its trivial extension.
Let $Q^E$ be the bound quiver of $\Lambda^E$.
We use the same notation for the vertex of $Q'$ and $Q^E$.
Take the second indices for vertices of $Q^t$ from $\mathbb Z$, and second indices for the vertices of $Q^E$  from $\mathbb Z/(l+1)\mathbb Z$.

Our theorem follows immediately from the following lemma.
\end{proof}

\begin{lemma}\label{mmain} For $t=0,1,\ldots, r$, $Q^t$ is a full bound subquiver of $Q^E$, $Q^E$ is obtained from $Q^t$  by adding an arrow $(i,n)$ to $(j,n+1-l)$ whenever there is a linearly independent bound path of length $l-1$ from $(j,n+1-l)$ to $(i,n)$, and $$\iota (\Lambda^t)=\iota (k(Q^t)) \simeq k(Q^E) =\Lambda^E.$$
\end{lemma}
\begin{proof}
We prove this lemma by induction on $t$.
Clearly, $Q^E$ is obtained from $Q^0=Q'$ by adding an arrow from vertices $(i,l-1)$ to $(j,0)$ whenever there is a linearly independent path of maximal length from $(j, 0)$ to $(i,l-1)$.
So the lemma holds when $t=0$.

Assume that $0 < t \le r$ and the lemma holds for $t-1$.

We may assume that $\sigma_t =s_{i_0}^-$, the other case is proved similarly.
There is a vertex $(i_0,m)$  such that $Q^{t}$ is obtained from $Q^{t-1}$ by removing the vertex $(i_0,m)$ with all the arrows ending at $(i_0,m)$
and adding the vertex $(\tau i_0,m-l)$ together with all the arrows starting from $(\tau i_0,m-l)$ in $\overline{Q}$.

Both $Q^{t-1}$ and $Q^t$  have the same set of vertices as $Q^E$
(after taking the image of their second indices in $ \mathbb Z/(l+1)\mathbb Z$).
So we may identify the degree zero part $$\Lambda^t_0 = \Lambda^{t-1}_0 = \Lambda^{E}_0.$$
The arrows of the two subquivers are the same except those concerning the vertices $(\tau i_0,m-l)$ and $(i_0,m)$.
$(\tau i_0,m-l)$ is a source in $Q^t$ and $(i_0,m)$ is a sink in $Q^{t-1}$.
Let $Q^{t-1,'}$ and $ Q^{t,'}$ be  the full bound quiver obtained from $Q^{t-1}$, and respectively $ Q^{t}$, by removing the vertex $(i_0,m)$, and respectively $(\tau i_0,m-l)$.
Then  $Q^{t-1,'}= Q^{t,'}$, denote it by $Q^*$.

By induction, $\Lambda^E = k(Q^E) \simeq \iota (\Lambda^{t-1})$, and $Q^E$ is obtained from $Q^{t-1}$ by adding an arrow $(i,n)$ to $(j,n-l+1)$ whenever there is a linearly independent bound path of length $l-1$ from $(j,n-l+1)$ to $(n,i)$.

Let $\Lambda^* = k(Q^*)$ be the algebra defined by the bound quiver $Q^*$, it is a subalgebra of both $\Lambda^{t-1}$ and $\Lambda^{t}$.
Denote by $\widetilde{\Lambda} =k( \overline{Q}, v)$ the algebra of the bound quiver $(\overline{Q},v) $.
It is a direct summand of $\Lambda\# k\mathbb Z^*$, so it is locally finite-dimensional self-injective algebra of Loewy length $l+1$.
$\Lambda^*, \Lambda^{t-1}$ and $\Lambda^{t}$ are embedded in $\widetilde{\Lambda}$ as subalgebras.

For the arrows ending at $(i_0,m)$ in $\iota (\Lambda^t)$, by Lemma \ref{bil}, we have that for any $(j,m-1)$
$$\arr{rllllll}{
& e_{(\tau i_0, m-l)} \iota (\Lambda^t)_{1}e_{(j,m-1)}   &\simeq & De_{(j,m-1)}\Lambda^{t}_{l-1}  e_{(\tau i_0,m-l)}  \\ \simeq & e_{(i_0,m)}\widetilde{\Lambda}e_{(j,m-1)} &= & e_{(i_0,m)}\Lambda^{t-1}_1 e_{(j,m-1)}\\ = & e_{(i_0,m)} \iota (\Lambda^{t-1})_1 e_{(j,m-1)} & \simeq & e_{(i_0,m)}\Lambda^E_1e_{(j,m-1)},}$$
as $\Lambda^t_0$-bimodules.
Note that $\Lambda^t_0 = \Lambda^E_0  $.
By identifying the vertex $(i_0, m)$ of $Q^{t-1}$ with the vertex $(\tau i_0,m-l)$ of $Q^t$, this implies that the number of the arrows from $(j,m-1)$ to $(i_0,m)$ in the quiver of $\iota(\Lambda^t)$ is the same as that of $Q^E$, and it equals the number of linearly independent bound paths of $Q^t$ of length $l-1$ from $(\tau i_0,m-l)$ to $(j,m-1)$.
Since $e_{(i,n)} \Lambda^t_{l-1} e_{(j,n+1-l)} = e_{(i,n)} \Lambda^{t-1}_{l-1} e_{(j,n+1-l)}$ for other pair $(j,n+1-l), (i,n)$ of vertices.
It follows from induction that $Q^E$ is obtained from $Q^t$  by adding an arrow $(i,n)$ to $(j,n+1-l)$ whenever there is a linearly independent bound path of length $l-1$ from $(j,n+1-l)$ to $(i,n)$.

Denote by $M$  the sub-$\Lambda^*$-bimodule of $\iota (\Lambda^t)$ generated by $ De_{(j,m-1)}\Lambda^{t}_{l-1}  e_{(\tau i_0,m-l)} $, it is isomorphic to $\bigoplus_{(j, n) \in Q^*_0} e_{( i_0,m)}\widetilde{\Lambda}e_{(j,n)} \simeq e_{(i_0,m)} \Lambda^{t-1}$.
Thus $\Lambda^{t-1}$ is isomorphic to a trivial extension
$$\Lambda^{t-1} \simeq (\Lambda^* + e_{( i_0,m)}\iota (\Lambda^{t})_0 e_{(i_0,m)}) \ltimes M,$$
and the right hand side is a subalgebra of $\iota (\Lambda^t)$.

%On the other hand,
%$$\arr{rllllll}{&   e_{(j,m+1)}  \Lambda^t_{1} e_{(\tau i_0,m-l)} &\simeq & e_{( j,m+1)}\widetilde{\Lambda}_{1}e_{(\tau i_0,m-l)}\\ \simeq  & e_{(j, m+1)}\iota (\Lambda^{t-1})_{1}  e_{( i_0,m)} &\simeq &  e_{(j,m+1)} (D\Lambda^{t-1}_{l-1})  e_{( i_0,m)}\\ \simeq &  e_{(j, m+1)} \Lambda^E_{1} e_{(i_0,m)} ,}$$ for all $j$, and it is the dual space of $e_{(i_0,m)}\Lambda^{t-1}_{l-1}e_{(j,m+1)}$.
Since $\Lambda^{t-1}$ and $\Lambda^t$ are graded algebras whose maximal bound paths having the same length $l-1$.
$\Lambda^t_{l-1} = \Lambda^*_{l-1} + e_{(i_0,l)}\Lambda^{t-1}_{l}$ and $\Lambda^{t-1}_{l-1}= \Lambda^*_{l-1} + \Lambda^{t-1}_{l-1}e_{(\tau i_0,0)}$,  $D \Lambda^{t-1}$ is  generated by $D \Lambda^{t-1}_{l-1} $ as $\Lambda^{t-1} $-bimodule,
and $D\Lambda^t$ is  generated by $D\Lambda^t_{l-1}$ as $\Lambda^t$-bimodule.

$$\Lambda^E \simeq \iota (\Lambda^{t-1}) = \Lambda^{t-1}+ D\Lambda^{t-1}$$
as $\Lambda^{t-1}$-bimodule with multiplication defined naturally (as the trivial extension).
So one gets
$$\iota (\Lambda^t) = \Lambda^* + \Lambda^t e_{(\tau i_0,m-l+1)} + D\Lambda^* + D\Lambda^{t}e_{(\tau i_0,m-l+1)}$$
as $\Lambda^*$-bimodule.
Note that $\Lambda^*  + D\Lambda^{t}e_{(\tau i_0,m-l+1)}$ is a subalgebra and
$$\arr{lll}{\Lambda^*  + D\Lambda^{t}e_{(\tau i_0,m-l+1)} &\simeq &\Lambda^*  + e_{(\tau i_0, m-l+1)} D\Lambda^{t}_{l}e_{(\tau i_0,m-l+1)} +M \\  &\simeq &(\Lambda^* + e_{( i_0,m)}\iota (\Lambda^{t})_0 e_{(i_0,m)}) \ltimes M\\ &\simeq&  \Lambda^{t-1},}$$
and $ \Lambda^t e_{(\tau i_0,m-l+1)} + D\Lambda^* $ is a $(\Lambda^* + e_{( i_0,m)}\iota (\Lambda^{t})_0 e_{(i_0,m)}) \ltimes M$-bimodule.
Identify the above isomorphic algebras, we get a $\Lambda^{t-1}$-bimodule isomorphism
$$ \Lambda^t e_{(\tau i_0,m-l+1)} + D\Lambda^* \simeq D \Lambda^{t-1}.$$
The multiplication of $\iota (\Lambda^{t})$ defines a trivial extension of
$(\Lambda^* + e_{( i_0,m)}\iota (\Lambda^{t})_l e_{(i_0,m)}) \ltimes M$ by its bimodule $ \Lambda^t e_{(\tau i_0,m-l+1)} + D\Lambda^* $.
Thus
$$\arr{lcl}{\iota (\Lambda^{t})& =&((\Lambda^* + e_{( i_0,m)}\iota (\Lambda^{t})_0 e_{(i_0,m)}) \ltimes M) \ltimes (\Lambda^t e_{(\tau i_0,m-l+1)} + D\Lambda^*) \\&\simeq& \Lambda^{t-1}\ltimes D\Lambda^{t-1} = \iota (\Lambda^{t-1})=\Lambda^E.}$$
This shows that our lemma holds for $\Lambda^t$.
\end{proof}

Denote by $\Lambda(i)$ the special truncate algebra defined by the bound quiver $(Q^N,i)$. As a corollary, we have the follow theorem.

\begin{thm}\label{tiso} $\iota (\Lambda(i))$ are isomorphic for all $i \in Q_0$. \end{thm}

According to \cite{h1,cx}, we have an equivalence between the bounded derived categories $\mathcal D^b(\Lambda')$ of the $\tau$-slice algebra $\Lambda'$ and $\underline{\mathrm gr}\, \iota (\Lambda')$ of the stable category of the finite-dimensional graded modules over its trivial extension.

Let $\Lambda$ be a graded self-injective algebra. As a corollary of Theorem ~\ref{trexs} and Corollary 1.2 of \cite{cx},  we get:

\begin{cor}\label{eqidc} Let $\Lambda$ be a graded self-injective algebra.
Then all the $\tau$-slice algebras of $\Lambda$ have equivalent bounded derived categories, and they are equivalent to the stable category of graded $\Lambda$-modules.
\end{cor}

%%w
Let $Q^S$ be a complete slice in $(\overline{Q}, i)$ and let $I^S[r] = \{(\tau^r i,n-lr) |(i,n) \in Q^S_0 \}$ for $r \in \mathbb Z$.
Let $e_{I^S[s]} = \sum_{(i,n)  \in I[s] } e_{(i,n)} $, and let $$M^S[s,t] = e_{I^S[s]}  \Lambda \# k \mathbb Z^* e_{I^T[t]} = \sum_{(i,n)  \in I^S[s]}\sum_{(j,m)  \in I^S[t] } e_{(i,n)} \Lambda \# k \mathbb Z^*  e_{(j,m)}$$
for all $s,t \in \mathbb Z$.
Let $\Gamma^S[s] = M^S[s,s]$.
Then $\Gamma^S[s]$ are isomorphic to the complete $\tau$-slice algebra $\Lambda^S = k(Q^S)$.
$M^S[s+1,s]$ are $\Gamma^S[s+1]$-$\Gamma^S[s]$-bimodules and $M^S[t,s]$ =0 for $t \neq s, s+1$, and $ M[s,t] M[s,t] = 0$ when $s \neq t$.

Similar to the proof of Theorem ~\ref{repetitive} and Proposition ~\ref{directsum}, we see that
$\left(\arr{ccccccc}{\ddots \\ & \Gamma^S[s+1] & M^S[s+1,s] & \\ & &\Gamma^S[s] & M^S[s,s-1]  \\ & & &\Gamma^S[s-1] &\\& & &&\ddots}\right) $ is the direct summand of $\Lambda \# k \mathbb Z^*$ corresponding to the component $(\overline{Q}, i)$.
Clearly it is isomorphic to the repetitive algebra $\widehat{\Lambda^S}$.
So we have the following theorem.

\begin{thm}\label{repetitiveslice}
The repetitive algebra of a $\tau$-slice algebra is a direct summand of $\Lambda \# k \mathbb Z^*$.

All the  $\tau$-slice algebras have the isomorphic repetitive algebras.
\end{thm}

Recall that a subcategory $\mathcal T$ of a triangulated category $\mathcal C$ is called a {\em tilting subcategory} if it generates $\mathcal C$ and we have $\mathrm{Hom}({\mathcal T,\mathcal T[i]}) = 0$ for all $i \neq 0$.
According to \cite{ric}, two algebra $\Lambda'$ and $\Lambda''$ of finite global dimension are derived equivalent if and only if  there is a tilting subcategory $\mathcal T = \mathrm{add} T$ of $\mathcal D \Lambda'$ for some objects $T$ and $\Lambda'' = \mathrm{End} T$.
Now let $\sigma_1,\cdots,\sigma_r$ be a sequence of $\tau$-mutations.
If $\Lambda'$ is a $\tau$-slice algebra, then $\sigma_r \cdots \sigma_1 \Lambda' $ is also a $\tau$-slice algebra and we have that $\iota (\Lambda') \simeq \iota (\sigma_r \cdots \sigma_1 \Lambda')$.
Thus $\mathcal D^b(\Lambda')$ and $\mathcal D^b( \sigma_r \cdots \sigma_1 \Lambda')$ are equivalent as triangulated category.
Hence by \cite{ric} there is a tilting object $T$ in $\underline{\mathrm gr}\, \iota (\Lambda')$ such that $\sigma_r \cdots \sigma_1 \Lambda' \simeq \mathrm{End}\, T$.
So we get:

\begin{cor}\label{tilt}
Let $\Lambda'$ and $\Lambda''$ be $\tau$-slice algebras of a graded self-injective algebra $\Lambda$, then there is a tilting object $T$ in $\underline{\mathrm gr}\, \Lambda$, such that $\Lambda' \simeq \mathrm{End}\, T$.
\end{cor}

\section{Koszul Duality and $\tau$-mutations}

It is natural to investigate the relationship between our $\tau$-mutation and the BGP reflection.
They are not related directly, as is shown in the following example.
Consider the case $l=2$.
Let $\Lambda$ be a self-injective algebra with vanishing radical cube whose quiver $Q$ is the double quiver of $A_5$.
It follows from \cite{g2} that this quiver is a stable bound quiver of Loewy length $3$ with trivial Nakayama translation.
Its separated directed quiver $\overline{Q}$ is two copies of $\mathbb Z A_5$, with the total special truncation $Q^N = (Q^N,1)\cup  (Q^N,2)$, where
$$\arr{cc}{(Q^N,1): & (1,0)\rightarrow (2,1) \leftarrow
(3,0 )\rightarrow (4,1) \leftarrow (5,0 )\\
(Q^N,2): & (1,1) \leftarrow (2,0)\rightarrow
(3,1 ) \leftarrow (4,0)\rightarrow  (5,1 ).
}$$
Let $\Lambda^S$ be the $\tau$-slice algebra of $Q^S = (Q^N,1)$.
Then its $\tau$-mutation $s^-_{(2,1)}\Lambda^S$ is given by the bound quiver
$$s^-_{(2,1)} Q^S : (1,0) \leftarrow (2,1)\rightarrow
(3,0 )\rightarrow (4,1) \leftarrow (5,0 )$$ with a single relation $(2,1)\rightarrow (3,0 ) \rightarrow (4,1)$.
If we do BGP reflection for the quiver $Q^S = (Q^N,1)$ at vertex $(2,1)$, one gets the same quiver without relation.

Note that the above algebras $\Lambda^S$ and $s^-_{(2,1)} \Lambda^S$ are both Koszul, and we can look at their Koszul dual, which are just the path algebras of $Q^S$ and $s^-_{(2,1)}Q^s$, respectively.
So we see that our $\tau$-mutation coincides with  the BGP reflection when applied to the Koszul dual of the $\tau$-slice algebras.
This is true in general.

Assume that $\Lambda$ is a  self-injective algebra with vanishing radical cube whose bound quiver is $Q$.
In this case $l=2$.
Let $Q^S$ be a complete $\tau$-slice, let $\Lambda^S$ be the $\tau$-slice algebra with bound quiver $Q^S$.
$\Lambda^S$ is an algebra  with vanishing radical square, so are its $\tau$-mutations.
$Q^S$ is a directed quiver, so its orientation is admissible.
Since $\Lambda\# k \mathbb Z^*$ is a self-injective algebra with vanishing radical cube, its bound quiver $\overline{Q}$ is a stable translation quiver \cite{g2} with  $\overline{\tau}$ as the translation.
Now let $(i,m)$ be a sink in $Q^S$, then  $\overline{\tau}(i,m) = (\tau i, m-2 )$ is not a vertex of $Q^S$.
We have that $H_{(i,m)}$ forms a mesh in $\overline{Q}$,
$$
 \arr{ccccc}{&&(j_1,m-1)&&\\ &\stackrel{\alpha_1}{\nearrow}& &\stackrel{\beta_1}{\searrow}\\ (\tau i, m-2 )&&\vdots& &( i, m )\\ &\stackrel{\alpha_r}{\searrow}& &\stackrel{\beta_r}{\nearrow}\\&&(j_r,m-1) &&}
$$
with $\beta_1, \ldots, \beta_r $ are arrows in $Q^S$ and  $\alpha_1, \ldots, \alpha_r $ not in $Q^S$.
The $\tau$-mutation $s^{-}_{(i,m)}(Q^S)$ is obtained from $Q^S$ by replacing $(i,m)$ with $(\tau i, m-2 )$ and each $\beta_t $ with $\alpha_t$.
The BGP reflection  acts on a quiver $Q$ at its sink  by just reverse all arrows to  this vertex.
When we identify the vertex $(i,m)$ with $(\tau i, m-2 )$, this is exactly the BGP reflection of quiver $Q^S$ at the sink $(i,m)$.
Same argument also works when $(i,m)$ is a source.

Since algebra with radical squared zero is Koszul whose Yoneda algebra is hereditary when its quiver dos not contain oriented cycle.
Thus their Yoneda algebras $E(\Lambda^S)$ and $E(s^{-}_{(i,m)}\Lambda^S)$ are hereditary algebras, with the quiver $Q^S$  and $s^{-}_{(i,m)}(Q^S)$, respectively, without relation.

Use the same notations $s^{-}_{(i,m)}$ and respectively $s^{+}_{(i',m')}$ for the BGP reflections on a path algebra at a sink $(i,m)$ and respectively at a source $(i',m')$ of the quiver.
We gave the following observation.

\begin{prop}\label{here}
Assume that $\Lambda$ is a  self-injective algebra with vanishing radical cube and $\Lambda^S$ be a $\tau$-slice algebra of $\Lambda$. Then the Yoneda algebra of a  $\tau$-mutation of $\Lambda^S$ at a vertex $(i,m)$ is the BGP reflection of the Yoneda algebra of $\Lambda^S$ at the same vertex. That is
$$ E(s^{\pm} (i,m) \Lambda^S) =s^{\pm} (i,m) E( \Lambda^S).
$$
\end{prop}

We remark that in the case of Dynkin quivers, the graded self-injective algebras we starting with are not Koszul, but all the $\tau$-slice algebras are Koszul since they have vanishing radical square.

\medskip

Now we consider the case of starting with a Koszul self-injective algebra.
Let $Q$ be the bound quiver of a graded self-injective algebra and $\overline{Q}$ be its separated direct quiver.
Now consider the orbit algebra $\Lambda^{T,r+1} = O((\mathcal N^{r+1}), P)$ where $P$ is a basic $(\mathcal N^{r+1})$-orbit generator of $\mathcal P(\Lambda\# k\mathbb Z^*)$.
It follows from Proposition ~\ref{gcv1} that $\Lambda^{T,(r+1)}$ is a finite regular covering of $\Lambda^T$. It is proved in \cite{zh} that a finite regular covering of a Koszul self-injective algebra is also Koszul.
So we have the following proposition.

\begin{prop}\label{gcv2}
If $\Lambda^T$ is a Koszul algebra, so is $\Lambda^{T,(r+1)}$.

\end{prop}

Assume that  $\Lambda^T$ is Koszul, we are going to prove that all the  $\tau$-slice algebras  are also Koszul. We need some preparation.

A subset $U$ of $\overline{Q}_0$ is called {\em convex} provided that  for any $(i,m), (j,n) \in U$, if there is a bound path from $(i,m)$ to $(j,n)$ such that all its vertices are in U, then for any bound path from $(i,m)$ to $(j,n)$, its vertices are all in $U$.
A full subquiver of $\overline{Q}$ is called {\em convex} if its vertex set is convex.

\begin{lemma}\label{convsub}
Let $Q'$ be a bound subquiver of a bound quiver $Q$, if $Q'$ is convex, then $k(Q')$ is a subalgebra of $k(Q)$.
\end{lemma}
\begin{proof}
Clearly, we have an embedding of path algebras $i: kQ' \to kQ$.
Since $Q'$ is convex, a relation of $Q$ is either a relation of $Q'$ or its terms contain no paths in $Q'$.
Thus $i$ induces an embedding from $k(Q')$ into $k(Q)$, and our assertion holds.
\end{proof}

An initial complete $\tau$-slice $(Q^N,i)$ is convex.
Apply  Lemma ~\ref{refl}, using induction on the number of $\tau$-mutations needed to reach a $\tau$-slice $Q^S$ from an initial complete $\tau$-slice $(Q^N,i)$, one easily obtains the following proposition.

\begin{prop}\label{convex}
A complete $\tau$-slice is convex.
\end{prop}

\begin{thm}\label{slck}
If $\Lambda$ is a self-injective algebra such that $\Lambda^T$ is Koszul.
Then each of it's $\tau$-slice algebra is a Koszul algebra with finite global dimension.
\end{thm}

\begin{proof} Let $\Lambda^S$ be a $\tau$-slice algebra with the bound quiver $Q^S$.
Assume that the depth of $Q^S$ is $d$.
We may assume that $min\{n |(i,n) \in Q^S_0\} =0$, by shifting suitably.
Let $r$ be a positive number such that $(r-1)l\le d <rl$. Now embed $Q^S$ into $Q^{r+1,T}$ as a full bound subquiver.
We show that $\Lambda^S$ is a subalgebra of $\Lambda^{r+1,T}$, the latter is Koszul, by Proposition ~\ref{gcv2}.
By  Lemma ~\ref{convsub}, $Q^S$ is convex. Let $e^S = \sum_{(i,n)\in Q^S_0}e_{(i,n)}$, then  by Proposition ~\ref{convex},
$\Lambda^S  = e^S \Lambda^{r+1,T}e^S \simeq \Lambda^{r+1,T}/\Lambda^{r+1,T}(1-e^S)\Lambda^{r+1,T} $
is both a subalgebra and a quotient algebra of $ \Lambda^{r+1,T}$.

Let $M$  be a $\Lambda^S$-module which is Koszul as $\Lambda^{r+1,T}$-module, and let
$$\cdots \longrightarrow P^{(t)} \slrw{f_t} \cdots \longrightarrow P^{(0)} \slrw{f_0} M \longrightarrow 0$$
be a minimal projective resolution of $M$ as a $\Lambda^{r+1,T}$-module.
Then $P^{(t)}$ is generated in degree $t$, apply the functor $\hmm{\Lambda^{r+1,T}}{\Lambda^{r+1,T}e^S,\mbox{ }}$, then we get an exact sequence
$$\arr{rrl}{\cdots \longrightarrow &\hmm{\Lambda^{r+1,T}}{\Lambda^{r+1,T}e^S,P^{(t)}}& \slrw{f_t} \cdots \\ \longrightarrow
&\hmm{\Lambda^{r+1,T}}{\Lambda^{r+1,T} e^S, P^{(0)}} &\slrw{f_0} M \longrightarrow 0.
}$$\normalsize
Assume that $P^{(t)} = \bigoplus_{(i,n+t)\in Q^{r+1,T}_0} m(i,n+t)\Lambda^{r+1,T}e_{(i,n+t)}$ for some nonnegative integer $m(i,n+t)$.
We see that $ P^{(t)}$ is generated at degree $t$, and the $t$th term has decomposition as graded vector space with respect to the idempotent $e^S$:
$$\arr{lll}{&& \hmm{\Lambda^{r+1,T}}{\Lambda^{r+1,T}e^S,P^{(t)}} \\ &\simeq &\bigoplus_{(i,n+t)\in Q^{r+1,T}_0} m(i,n+t) e^S \Lambda^{r+1,T}e_{(i,n+t)}\\&=& \bigoplus_{(i,n+t)\in Q^{r+1,T}_0} m(i,n+t) e^S \Lambda^{r+1,T}e_{(i,n+t)}e^S \\ && \oplus \bigoplus_{(i,n+t)\in Q^{r+1,T}_0} m(i,n+t) e^S \Lambda^{r+1,T}e_{(i,n+t)}(1-e^S)\\ &=& \bigoplus_{(i,n+t)\in Q^{S}_0} m(i,n+t) e^S \Lambda^{r+1,T}e_{(i,n+t)} \\ && \oplus \bigoplus_{(i,n+t)\not\in Q^{S}_0} m(i,n+t) e^S \Lambda^{r+1,T}e_{(i,n+t)}.}$$
Since $Q^S$ is convex in $Q^{r+1,T}$, $ e^S \Lambda^{r+1,T}e^S \simeq \Lambda^{S}$, and
$$ \arr{rcl}{\cdots \longrightarrow & \bigoplus_{(i,n+t) \in Q^{S}_0} m(i,n+t) e^S \Lambda^{r+1,T}e_{(i,n+t)} & \slrw{f'_t} \cdots\\ \longrightarrow & \bigoplus_{(i,n+t)\in Q^{S}_0} m(i,n) e^S \Lambda^{r+1,T}e_{(i,n)} &\slrw{f'_0} M \longrightarrow 0 }$$
is a projective resolution of $M$ as a $e^S\Lambda^{r+1,T} e^S\simeq \Lambda^S$-module.
And we have that $\bigoplus_{(i,n+t)\in Q^{S}_0} m(i,n+t) e^S \Lambda^{r+1,T}e_{(i,n+t)}$ is generated at degree $t$.
This shows that $M$ is Koszul as a $\Lambda^S$-module.
Especially $\Lambda^S$ is Koszul algebra when $\Lambda^T$ is so.
\end{proof}

We also know that the global dimension of a special truncated algebra is $l-1$, so we have the following result.

\begin{cor}\label{trk}
A special truncated algebra of a graded self-injective Koszul algebra of Loewy length $l+1$ is Koszul algebra of global dimension $l-1$.
\end{cor}

\section{An Example: $m$-cubic Algebra and Absolute $m$-complete Algebra}

In the previous sections we discuss $\tau$-slices, which are truncations   related to the graded self-injective algebras and Nakayama translation $\tau$.
It follows that  hereditary algebras can be obtained in this way, they are $\tau$-slices of the Koszul dual of graded self-injective algebras of Loewy length $3$.

Iyama introduces $n$-complete and absolute $n$-complete algebras in \cite{iy3}.
He also describe absolute $n$-complete algebra by quiver with relations.
In this section we display how Iyama's absolute $n$-complete algebra is recovered as a truncation   related to the Koszul dual of self-injective algebra and Nakayama translation.

Assume that $k$ is an algebraically  closed field of characteristic $0$. Let $V(m)= k^m $ be an $m$-dimensional vector space over $k$.
Let  $\wedge V(m)$ and $ k[V(m)]$ be the exterior algebra and the symmetry algebra, respectively, defined on $V(m)$.
Let $G(m) = (\mathbb Z/(r+1)\mathbb Z)^m$ be a finite Abelian group which is the direct sum of $m$-copies of cyclic group of order $r+1$.
We then show that  Iyama's  absolute $m$-complete algebra is related to our algebras as certain truncation related to the Nakayama translation.
This can be regarded as a higher dimension analog of the well known fact that Dynkin quiver (quiver of representation finite algebra) is a truncation of extended Dynkin quiver (quiver of representation tame algebra) by removing one vertex.

In \cite{g08}, we show that the vertex set of the McKay quiver of $G(m) = (\mathbb Z/(r+1)\mathbb Z)^m$ is indexed by the set $ (\mathbb Z/(r+1)\mathbb Z)^m$.
Each standard basic element $\epsilon_i = (0, \ldots, 0, 1, 0, \ldots, 0)$ provides an arrow from $v=(i_1,\ldots,i_m)$ to $v+\epsilon_i$ for each vertex $v$.
By \cite{gum}, the McKay quiver $Q(m)$ of $G(m)$ is the Gabriel quiver of the skew group algebra $\wedge V(m) * G(m)$, with arrows afforded by a basis of $V(m)$.
We also denote by $Q(m)$ the bound quiver of $\wedge V(m) * G(m)$.
Set $\epsilon_s =(0, \ldots, 0,1 ,0 \ldots, 0)$, the $m$-dimensional unit vector with only nonzero element $1$ at the $s$th position.
The following proposition restates this fact.
We provide a proof which also shows how the new arrows appear as the dimension growing.

\begin{prop}\label{cubicqv} There is an embedding of $G(m) = (\mathbb Z/(r+1)\mathbb Z)^m$ in $\mathrm{GL}(m,k)$ such that the McKay quiver $Q=Q(m)$ of $G(m)$  has vertex set $Q_0=(\mathbb Z/(r+1)\mathbb Z)^m$ and arrow set $\{\alpha_{i,v}: v \to v+\epsilon_i|v \in Q_0, i = 1, \ldots, m\}$, with Nakayama translation $\tau(m)$  sending each vertex $v$ to $v-(1,\ldots,1)$.\end{prop}

\begin{proof} We prove this proposition by induction on $m$.

Consider cyclic subgroup $G(1)= \mathbb Z/(r+1)\mathbb Z$ of order $r+1$ of $k^* = \mathrm{GL}(1,k)$.
By \cite{gcv}, the McKay quiver $Q(1)$ of $G(1)$ is an affine quiver of type $\widetilde{A}_r$ with cyclic orientation.
Its vertices are marked by the elements in $\mathbb Z/ (r+1)\mathbb Z$.
Let $V(1)=k$ be a one-dimensional vector space with basis $x_1$.
The arrows of $Q(1)$ are afforded by $x_1$, with one at each vertex $i$, $\alpha_i: i \to i+1$.
Since $\wedge V(1)$ is an algebra with vanishing radical square, the same holds for $\wedge V(1)*G(1)$.
So the relations are just paths of length $2$ and the Nakayama translation $\tau(1)$ sends $i$ to $i-1$.

Now assume that $m > 1$ and  the proposition holds for $m-1$.
We also assume that there is a flag of subspaces
$$V(1)\subset V(2) \subset \cdots \subset V(m-1),$$
such that $V(t)$ is a $t$-dimensional subspace and $G(t) \subset \mathrm{GL}(t,k)$ acts on $V(t)$ naturally.
Let $x_1, \ldots, x_{m-1}$ be a basis of $V(m-1)$ such that $x_t \in V(t)$.
$G(m-1)$ is a subgroup of  $\mathrm{GL}(m-1, k) $.
Its McKay quiver $Q(m-1)$ is the bound quiver of the skew group algebra $\wedge V(m-1)*G(m-1)$,
and $x_t$ affords the arrow $v \to v+\epsilon_t$ for $t=1, \ldots, m-1$.
The Nakayama translation $\tau(m-1)$ defined by the Nakayama functor of $\wedge V(m-1)*G(m-1)$ sends $v$ to $v-\sum_{t=1}^{m-1} \epsilon_t$

By \cite{gcv}, we have an embedding of vector space $V(m-1)$ in $V(m)$ with an additional basic element $x_m$.
This leads to a natural embedding of $G(m-1)$ in  $\mathrm{SL}(m,k)$ with image $G'(m)$, whose McKay quiver $Q'(m)$ has the same vertex set as  $Q(m-1)$.
$Q'(m)$ is obtained from  $Q(m-1)$ by adding an arrow from $v$ to $\tau({m-1}) v$ for each vertex $v$.
The new arrows are  afforded by $x_m$ and the Nakayama translation $\tau'(m)$ of $Q'(m)$ is trivial.

Choose an element $\zeta_{m}$ of order $r+1$ in the center of $\mathrm{GL}(m,k)$ such that $(\zeta_{m}) \cap \mathrm{SL}(m,k) = \{1\}$.
For example, take $\zeta_{m} $ as the diagonal matrix with all diagonal elements $1$ except for the last one, which is taken as an $(r+1)$th primitive root of the unit.
The subgroup $ G(m) $ of $\mathrm{GL}(m,k)$ generated by $G'(m) $ and $\zeta_{r+1}$ is isomorphic to $(\mathbb Z/ (r+1)\mathbb Z)^{m}$. And  $G(m)\cap \mathrm{SL}(m,k) = G'(m)$.

Since $G'(m)$ is a direct summand of $G(m)$, its irreducible representations are extendible to $G(m)$.
By Theorem 1.2 of \cite{gcv}, the McKay quiver $Q(m)$ of $G(m)$ is a covering of that of $G'(m)$ with the group isomorphic to $G(m)/G'(m)  \simeq \mathbb Z/ (r+1)\mathbb Z$.
It follows  that the vertex set of $Q(m)$  is $Q_0(m-1) \times (\zeta_m)$.
We identify the vertex set of $Q(m)$ with $(\mathbb Z/(r+1)\mathbb Z)^m$, with $(0,\ldots,0,t)$ for $\zeta_m^t$.
The arrows afforded by  $x_1, \cdots, x_{m-1}$ are within in each copy $Q(m-1) \times \zeta_m^t =Q(m-1) \times t $ of $Q(m-1)$ in $Q(m)$.
The new basic element $x_m$ affords a new arrow from $(v,t) $ to $(\tau({m-1}) v, t+1)$ for each vertex  $(v,t) $ of $Q(m)$, here $v \in Q_0(m-1)$ and $t\in \mathbb Z/(r+1)\mathbb Z$.
The arrows afforded by $x_m$ can be regarded as obtained by replacing the new arrow in $Q'(m)$ of $G'(m)$ from the vertex $v$ to $\tau'(m) v$ by an arrow from $(v,t) $ to $(\tau'({m-1}) v, t+1)$ for each $t$.

Re-indexing the vertex set by replacing $(a_1, \ldots, a_{m-1}, a_m)$ with $$(a_1, \ldots, a_{m-1}, a_m) +(a_m, \ldots, a_m, 0).$$
Then for each vertex $w \in Q_0(m)$, there is exactly one arrow in the McKay quiver $Q(m)$ of $G(m)$ going from $w$ to $w + \epsilon_s$ for each $s=1, \ldots, m$.
In particular, the arrow afforded by the new basic element $x_m$ in the additional dimension goes from $w$ to $w+\epsilon_m$.

Now we determine the Nakayama translation $\tau(m)$ for $Q(m)$.
It is known  the Nakayama translation are determined by the bound paths of the maximal length of $\wedge V(m)*G(m)$.
By \cite{gum}, such paths are just paths formed by arrows afforded by all the different basic elements.
Since each dimension affords an arrow from $w$ to $w+\epsilon_s$, so maximal path goes from $w$ to $w+(1,1,\ldots,1)$ and we have $\tau({m})^{-1} w = w + (1,1,\ldots,1)$, so  $\tau(m) w = w - (1,1,\ldots,1)$.

By induction, this proves our proposition. \end{proof}

It follows from \cite{gum} that the relations of $\wedge V(m)$ afford the relations of $\wedge V(m)*G(m)$.
Denote by $\alpha(v,t)$ the arrow in $Q(m)$ from the vertex $v$ to $v+\epsilon_t$.
It follows that the relations for $\wedge V(m)*G(m)$ are of the form $$\alpha(v+\epsilon_t,t) \alpha(v,t) \mbox{ and }\alpha(v+\epsilon_t,s) \alpha(v,t) + \alpha(v+\epsilon_s,t) \alpha(v,s)$$ for all $v \in Q_0(m)$ and $1 \le s,t \le m$.
We see easily that the following proposition holds.

\begin{prop}\label{cubic-h}
A $\tau$-hammock $H^v=H_{\tau^{-1}_m v}$ starting with vertex $v$ in $Q(m)$ is an $m$-dimensional cube with vertex
$$\{w|v \le w \le v+(1,1,\ldots,1)\}.$$
\end{prop}

Let $Q=Q(m)$, and let $\overline{Q}$ be its separated directed quiver.
We now describe an initial $\tau$-slices in the quiver $\overline{Q}$ containing the vertex $(i_1,\ldots,i_m, 0)$ in the following proposition.

\begin{prop}\label{cubslice}
The full subquiver of $\overline{Q}$ with the vertex set
$$\{(i_1+a_1,\ldots,i_m+a_{m},l)|0\le l \le m-1, a_i\in \mathbb Z, -m<a_i<m, \sum_{t=1}^m a_t = l \}$$
is an initial $\tau$-slice.
\end{prop}

The proof is direct and we omit it. It is interesting to describe the other $\tau$-slices, especially those with maximal depth.

\medskip

It is known that $\Lambda(m)=\wedge V(m)*G(m)$ is a self-injective Koszul algebra of complexity $m$.
Its Koszul dual is $\Gamma(m) = k[x_1,\ldots,x_m] * G(m) $.
$\Gamma(m)$ is an Artin-Schelter regular algebra with the same quiver, we call this algebra {\em $m$-cubic algebra} and denote its bound quiver by ${Q}^*(m)$.
The vertex set and the arrow set of ${Q}^*(m)$ are the same as those of $Q(m)$.
We have the following proposition describing its relations.

\begin{prop}\label{cbrelation}
The relation set of  $Q^*(m)$ is
$$\{\alpha(v+\epsilon_t,s)\alpha(v,t) - \alpha(v+\epsilon_s,t) \alpha(v,s)| v \in Q_0(m),1 \le s, t \le m \}.$$
\end{prop}
\begin{proof} This follows from the fact that $\Gamma(m)$ is the quadratic dual of $\wedge V(m)*G(m)$, and the description of the relations of $\wedge V(m)*G(m)$ below Proposition ~\ref{cubicqv}.
\end{proof}

In \cite{iy3}, Iyama introduces and studies $n$-complete algebra.
Let $\Lambda$ be a finite dimensional algebra with $\mathrm{gl.dim} \Lambda \le n$, and let $\tau_n = D \mathrm{Ext}_{\Lambda}^n (\quad , \Lambda)$ be the $n$-Auslander-Reiten translation.
Let $$\mathcal M =\mathrm{add}\{\tau_n^t D \Lambda | t \ge 0\},$$ $$\mathcal P=\mathcal P(\mathcal M) =\{X \in \mathcal M | \tau_n X=  0\}$$ and $$\mathcal M_{\mathcal P} =\{X \in \mathcal M | X \mbox{ has no} \mbox{ nonzero}\mbox{ summand}\mbox{ in } \mathcal P \}.$$
$\Lambda$ is called an {\em $n$-complete algebra} if the following conditions are satisfied.
\begin{enumerate}
\item There exists a tilting $\Lambda$-module $T$ satisfying $\mathcal P(\mathcal M) = \mathrm{add} T$,

\item $\mathcal M$ is an $n$-cluster tilting subcategory of $T^{\perp}$,

\item $\mathrm{Ext}_{\Lambda}^t (\mathcal M_{\mathcal P} , \Lambda)=0$ for any $0 < t < n$.
\end{enumerate}
$\Lambda$ is called  {\em absolute $n$-complete } if $\mathcal P(\mathcal M) = \mathrm{add} \Lambda$.
In \cite{iy3},  absolute $n$-complete algebras are classified.
Now we show how an $m$-cubic algebra and Iyama's absolute $n$-complete algebra $T_r^{(m)}(k)$ over $k$ are related.

\medskip
The quiver of $T_r^{(m)}(k)$ is given in \cite{iy3} as follow.

For $1 \le i \le r$, let $$\Delta_i=\Delta_i^{m-1}=\{(h_1,\ldots,h_{m-1}) \in \mathbb Z^{m-1}| h_1,\ldots,h_{m-1} \ge 0, h_1 + \ldots + h_{m-1} \le r-i \}.$$
For $1 \le t \le m-1$, put $v_t= \left\{\arr{ll}{-\epsilon_1&t=1\\ \epsilon_{t-1}-\epsilon_t& 1<t \le m-1}\right.$
The vertex set of ${Q}^{(m)}$ is
$${Q}^{(m)}_0 = \{(i, \mathbf h)| 1 \le i \le r, \mathbf h \in \Delta_{i} \}.$$
The arrow set ${Q}^{(m)}_1$ of ${Q}^{(m)}$ is
$$\arr{rl}{\{(i,\mathbf h ) \to (i+1,\mathbf h)|1\le i \le r-1, (i,\mathbf h ),(i+1,\mathbf h) \in {Q}^{(m)}_0 \}&\cup\\ \{(i,\mathbf h ) \to (i-1,\mathbf h+v_1)|1< i \le r, (i,\mathbf h ),(i-1,\mathbf h+v_1) \in {Q}^{(m)}_0 \}&\cup\\ \{(i,\mathbf h ) \to (i,\mathbf h +v_t)|1\le i \le r-1, (i,\mathbf h ),(i,\mathbf h +v_t) \in {Q}^{(m)}_0, 1<t \le m-1\}
.}$$
The $m$-Auslander-Reiten are given by $\tau_m (i,\mathbf h) = (i,\mathbf h)+\epsilon_{m-1}$, when both are vertices of ${Q}^{(m)}$.

The following proposition follows from Theorem 1.18, 1.19 and 6.22 of \cite{iy3}.
\begin{prop}\label{qwranc}
${Q}^{(m)}$ is the Gabriel quiver of the absolute $n$-complete algebra $T_r^{(m)}(k)$ with relations given by the commutative relation for each small square and zero relation for  for each small half square.
\end{prop}

We use the same ${Q}^{(m)}$ for this bound quiver of $T_r^{(m)}(k)$.
Now reindex the set of vertices of ${Q}^{(m)}$ with
$${Q}^{(m)}_0 = \{(h_0,h_1,\ldots,h_{m-1}) \in \mathbb Z^{m}|h_0 \ge 1, h_1,\ldots,h_{m-1}\ge 0, \sum_{t=1}^{m-1} h_t \le r-h_0 \}.$$
Write $Q^{(m)}_0(t) = \{(h_0,h_1,\ldots,h_{m-1}) \in Q^{(m)}_0, h_{m-1}=t\} $.
We immediately have the following fact.

\begin{prop}\label{qvcons}$Q^{(m)}_0= \cup_{t=0}^{r-1} Q^{(m)}_0(t), $ and  $Q^{(m)}_0(t)$ is obtained from  $Q^{(m)}_0(t-1)$ in such a way that for each vertex  $w\in Q^{(m)}_0(t)$,  $\tau_m w \in Q^{(m)}_0(t-1)$ and  there is an arrow $w \to w'$ with $w' \in  Q^{(m)}_0(t-1)$.
\end{prop}

Put $v_0= -\epsilon_0$ and $v_t= \epsilon_{t-1}-\epsilon_t$ for $t=1,\ldots,m$. Take $v_0, v_1,\ldots, v_{m-1}$ as basis and write the coordinates of the vertices with respect to this basis as $(u_0, u_1, \ldots, u_{m-1})$. We see that the arrows are given by
$$(u_0,u_1,\ldots,u_{m-1})\to (u_0,u_1,\ldots,u_{m-1})+v_t$$
for $(u_0,u_1,\ldots,u_{m-1}) \in {Q}^{(m)}_0 $ and $t=0, 1, \ldots, m$. Note that  $\epsilon_{m-1}=-v_0 - v_1- \ldots - v_{m-1},$, we also have
$$\arr{rcl}{\tau_m(u_0,u_1,\ldots,u_{m-1})& =& (u_0,u_1,\ldots,u_{m-1})- v_0 - v_1 - \cdots -v_{m-1}\\ & = &(u_0,u_1,\ldots,u_{m-1}) -(1,1,\ldots,1).}$$

Thus by Proposition \ref{cubicqv}, we have the following proposition.

\begin{prop}\label{emdebq}
${Q}^{(m)}$ is embedded into $Q^*(m)$ as a subquiver such that $m$-Auslander-Reiten translation coincides with the Nakayama translation when it is defined.
\end{prop}

So we can truncate a quiver $Q^{(m)}$ from  $Q^*(m)$. To show how to truncate $Q^{(m)}$ from $Q^*(m)$, we need the following concepts. Let $B$ be a subquiver of  $Q^*(m)$, denote by $B^-$ the set of vertices from which there is an arrow to some vertex of $B$.
Define the {\em set  $c'(B)$ of $\tau(m)$-induced vertices} of $B$  to be the subset
$$c'(B)= \{v \in B^{-}| \tau^{-1} v \in B_0 \}$$
of $Q^*_0(m)$.
Write $c{(0)} (B) =B_0$, and define inductively  $c{(t+1)} (B) =c'(c{(t)} (B))$.
Define the {\em $\tau(m)$-completion} $\widetilde{B}$ of $B$ to be the full subquiver with the vertex set $\cup_t c{(t)} (B)$.

Note that $Q^*(m)$ is a bound quiver with relations defined in Proposition \ref{cbrelation} and Nakayama translation $\tau=\tau({m})$,  we define the truncation $\widehat{Q}_{r}^{(m)}$ of  $Q^*(m)$ inductively as follow.

Define $\widehat{Q}_{r}^{(1)}$ to be the bound subquiver of  $Q^*(1)$ obtained  by removing the vertex $r+1$. Clearly, this is a Dynkin quiver of $A_r$ with linear order, so it is just $Q^{(1)}$, the bound quiver of $T_r^{(1)}(k)$

Assume that $\widehat{Q}_{r}^{(m-1)}$ for  $Q^*(m-1)$ is defined.
$Q^*(m-1)$ is embedded in $Q^*(m)$ as the full bound subquiver, say, with the vertex set
$$\{(u_1,\ldots,u_{m-1},r)|(u_1,\ldots,u_{m-1})\in Q^*_0(m-1)\}.$$
Under this embedding, regard $\widehat{Q}_{r}^{(m-1)}$ as a bound subquiver $\widehat{Q}_{r}^{m-1,'}$  of $Q^*(m)$.
Define $\widehat{Q}_{r}^{(m)}$ as the $\tau({m})$-completion of $\widehat{Q}_{r}^{m-1,'}$ in $Q^*(m)$.
We have the  following theorem.
\begin{thm}\label{absmc}
$\widehat{Q}_{r}^{(m)}$ is the bound quiver of the absolute $m$-complete algebra $T^{(m)}_r(k)$. \end{thm}
\begin{proof}
It follows from induction and Proposition ~\ref{qvcons} that ${Q}^{(m)}$ is just $\widehat{Q}_{r}^{(m)}$.
By Proposition \ref{cbrelation} and \ref{qwranc}, the relations are the same.
\end{proof}

It follows from our construction that $\widehat{Q}_{r}^{(m)}$ is connected for all $m$.
It is clear that  $\widehat{Q}_{r}^{(m)}$ is not unique, $Q^{(m)}$ may be embedded in $Q^*(m)$ in many ways.
If we remove all the vertices whose coordinates contains $r+1$ from  $Q^*(m)$, we get a $m$-cube with each side of length $r$, in the above embedding, $\widehat{Q}_{r}^{(m)}$ is in this cube with vertex set
$$\{(u_1, \ldots, u_{m})|r \ge  u_1 \ge u_2 \ldots \ge u_{m} \ge 1\}.$$

Let $E(m)$ be the full subquiver with vertex set $Q^*(m)_0\setminus (\widehat{Q}^{(m=
+1)}_r)_0$ and call it an {\em exceptional quiver in $Q^*(m)$.}
So we see that $\widehat{Q}_{r}^{(m)}$ is obtained by removing an exceptional quiver from $Q^*(m)$.
We see that the quiver of an absolute $m$-complete  algebra is obtained from an $m$-cubic McKay quiver by removing the exception set of vertices.
When $m=1$, it is classical that the quiver of a representation finite hereditary algebra of type $A_r$ is obtained from the quiver of a representation tame hereditary algebra of type $\widetilde{A}_r$ by removing an exceptional vertex.
When $m\ge 2$, things are much more complicated and we don't know how to find the set of exceptional vertices before we find $\widehat{Q}^{r}_{(m)}$.

\section*{Acknowledgements}
The author deeply appreciates the referees for suggestions and comments on revising and improving the paper.
The referees also bring the author's attention to the interrelation of the separated directed quiver, the smash product $\Lambda\# k\mathbb Z^*$ and the repetitive algebra of the Beilinsion algebra.
The proofs of Theorem ~\ref{lsquiver}, Theorem ~\ref{repetitive} and Theorem ~\ref{repetitiveslice} are due to the referees suggestion.

\end{document}